\documentclass[12pt,a4paper]{amsart}
\usepackage{amsmath,amssymb,amsfonts}
\usepackage[all]{xy}
\usepackage{caption}
\usepackage{enumerate}
\usepackage{hyperref} 
\usepackage{bm}
\usepackage{graphicx}
\usepackage{xcolor}
\usepackage{multicol}
\usepackage{tabularx}

\usepackage{soul}

\usepackage{mathpazo}
\usepackage{hyperref}
\usepackage{appendix}
\usepackage{a4wide}

\theoremstyle{plain}
\newtheorem{thm}{Theorem}[section]

\newtheorem{lem}[thm]{Lemma}
\newtheorem{cor}[thm]{Corollary}
\newtheorem{prop}[thm]{Proposition}

\newtheorem{conj}[thm]{Conjecture}

\theoremstyle{definition}
\newtheorem{definition}[thm]{Definition}
\newtheorem{ex}[thm]{Example}

\newtheorem{rmk}[thm]{Remark}

\makeatletter
\newcommand{\neutralize}[1]{\expandafter\let\csname c@#1\endcsname\count@}
\makeatother

\newcommand{\Gal}{{\rm Gal}}

\newcommand{\F}{{\mathbb F}}

\newcommand{\Q}{{\mathbb Q}}
\newcommand{\R}{{\mathbb R}}

\newcommand{\OO}{{\mathcal O}}

\newcommand{\Z}{{\mathbb Z}}

\newcommand{\ord}{\text{\rm ord}}

\begin{document}
\title{Fekete polynomials, quadratic residues, and arithmetic}
\dedicatory{Dedicated to Professor Paulo Ribenboim with gratitude and admiration}
\date{}

 \author{ J\'an Min\'a\v{c}, Tung T. Nguyen, Nguy$\tilde{\text{\^{e}}}$n Duy T\^{a}n }
\address{Department of Mathematics, Western University, London, Ontario, Canada N6A 5B7}
\email{minac@uwo.ca}
\date{\today}

 \address{Department of Mathematics, The University of Chicago, Chicago, Illinois, USA.}
 \email{tungnt@uchicago.edu}
 
  \address{
 School of Applied  Mathematics and 	Informatics, Hanoi University of Science and Technology, 1 Dai Co Viet Road, Hanoi, Vietnam } 
\email{tan.nguyenduy@hust.edu.vn}

\thanks{JM is partially supported  by the Natural Sciences and Engineering Research Council of Canada (NSERC) grant R0370A01. He gratefully acknowledges the Western University Faculty of Science Distinguished Professorship 2020-2021. NDT is funded by Vingroup Joint Stock Company and  supported by Vingroup Innovation Foundation (VinIF) under the project code VINIF.2021.DA00030
}

\begin{abstract}
Fekete polynomials associate with each prime number $p$  a polynomial with coefficients $-1$ or $1$ except the constant term, which is 0. These coefficients reflect the distribution of quadratic residues modulo $p$. These polynomials were already considered in the 19th century in relation to the studies of Dirichlet $L$-functions.  In our paper, we introduce two closely related polynomials. We then express their special values at several integers in terms of certain class numbers and generalized Bernoulli numbers. Additionally, we study the splitting fields and the Galois group of these polynomials. In particular, we propose two conjectures on the structure of these Galois groups. We also provide some computational evidence toward the validity of these conjectures.

\end{abstract}
\maketitle

\section{Introduction}
Recall that for each prime $p$ the Fekete polynomial $F_p(x)$ is defined by 
\[ F_p(x)=\sum_{a=1}^{p-1} \left(\frac{a}{p} \right) x^a .\] 
Here $\left(\dfrac{a}{p}\right)$ is the Legendre symbol which is equal to, for $1\leq a\leq p-1$, the value $1$ if $a$ is a quadratic residue modulo $p$ and the value -1 if $a$ is not a quadratic residue modulo $p$. Because $F_2(x)$ is just $x$, in our subsequent considerations we assume that $p$ is an odd prime number. 

The function $\chi_p(a)= \left(\dfrac{a}{p}\right)$, $a\in \Z$, is a Dirichlet character. The infinite series 
\[
L(s,\chi_p)=\sum_{n=1}^{\infty} \dfrac{\chi_p(n)}{n^s}, \; s=\sigma+it, \sigma,t\in \R,
\]
is an absolutely convergent for $\sigma>1$. It is well-known that $L(s,\chi_p)$ has an analytical continuation to the entire complex plane and this analytical continuation which we also denote as $L(s,\chi_p)$ is a regular function for all complex $s$ (see \cite[Chapter 12]{[Apostol]}). Michael Fekete observed that if $F_p(x)$ has no real zeroes in the interval $0<x<1$ then $L(s,\chi_p)$ has no real zero $s>0$. Therefore, these polynomials $F_p(x)$ are now called the Fekete polynomials. In fact, Fekete conjectured in 1912 that $F_p(x)$ has no real roots between $0$ and 1. George P\'olya in 1919 showed that this conjecture is false for $p=67$ and for infinitely many other primes (see \cite{[Po1]}, \cite{[Po2]}, \cite{[Alex]}.) Polya's examples of $f_p(x)$ with negative values on the interval $[0,1]$ use quadratic reciprocity law and Dirichlet's theorem on primes in arithmetic sequences to establish infinitely many primes $p$
 such that 
 
\[ \left(\frac{2}{p} \right)= \left(\frac{3}{p} \right)=\left(\frac{5}{p} \right)=\left(\frac{7}{p} \right)=\left(\frac{11}{p} \right)=\left(\frac{13}{p} \right)=-1.\]
A simple but elegant argument establishes the existence of $x_0 \in (0,1)$ such that for primes $p$ considered above $f_p(x_0)<0.$ These simple and beautiful considerations were also used in \cite[Problem 46 of part 5]{[Po3]}. Among some further interesting investigations of behaviors of $f_p(x)$ over the interval $[0,1]$ and also its roots in the complex plane, we mention just \cite{[Bateman]},\cite{[Conrey]}. Our main interest in these polynomials lies in exploring their arithmetic and Galois theoretic properties using the interesting interplay between the distribution of quadratic residues, arithmetic properties of Bernoulli numbers, class number formulas, elementary polynomials, and Galois theoretic considerations. This paper is an outgrowth of our further reflections on the special values of $L$-functions considered in \cite{[MTT1]}. Nevertheless, formally our paper is independent of \cite{[MTT1]}. 

We refer \cite[page 231]{[Lemmermeyer]} for some interesting historical comments on Fekete polynomials. In particular, Fekete polynomials already implicitly showed up in Gauss’s sixth proof \cite{[Gauss]} of the quadratic reciprocity law.

The structure of our article is as follows. In Section 2, we focus on determining the multiplicity of the roots $x=1$ and $x=-1$ of $F_p(x)$. Using some of the results in Section 2, we define in Section 3 the  polynomials $f_p(x)$ where we divide $F_p(x)$ by the factor $x(1-x)$ if $p \equiv 3 \pmod{4}$ and by $x(1-x)^2(x+1)$ if $p \equiv 1 \pmod{4}$. We observe further that $f_p(x)$ is a reciprocal polynomial of even degree. From this observation, we define another key polynomial $g_p(x)$ which is closely related to $F_p(x)$ and $f_p(x).$ We call this $g_p(x)$ the ``reduced Fekete polynomial" associated with $p$. These reduced Fekete polynomials are interesting as they contain considerable arithmetic information. In particular, we show that their special values $g_p(-2), g_p(-1), g_p(0), g_p(1), g_p(2) $ are closely related to the class numbers of some some specific quadratic number fields. In Section 4, we investigate some Galois theoretic properties of the splitting fields of $f_p(x)$ and $g_p(x)$ over the rational number field. Based on our results, heuristic considerations, and  numerical evidence, we make the Conjecture \ref{conj1} about the Galois group of $g_p(x)$. In Section 5, we investigate some modular properties of $F_p(x), f_p(x), g_p(x)$ modulo $p$. We conclude our paper with some specific numerical examples of $f_p(x)$ and $g_p(x).$
 

\section{Roots of $F_p(x)$} 

\subsection{Bernoulli numbers and generalized Bernoulli numbers}

We first recall that the Bernoulli polynomials $B_n(x)$ and Bernoulli numbers $B_n$, $n\geq 0$, are defined as  (see (\cite[pp. 7-9]{[Iwasawa]}).)
\[
\dfrac{te^{xt}}{e^t-1}=\sum_{n=0}^\infty B_n(x) \dfrac{t^n}{n!}
\]
and 
\[
B_n(x) =\sum_{k=0}^n \binom{n}{k} B_k x^{n-k}, \quad n\geq 0.
\]
For example, $B_0(x)=1$, $B_1(x)=x+\dfrac{1}{2}$, $B_2(x)=x^2+x+\dfrac{1}{6}$.

 One can define the generalized Bernoulli numbers $B_{n,\chi}$  and Bernoulli polynomials  $B_{n,\chi}(x)$, $n\geq0$, for a  Dirichlet character $\chi$ with conductor $f=f_\chi$ as (see \cite[pages 8-9]{[Iwasawa]})
 \[
 \sum_{a=1}^f \dfrac{\chi(a)te^{(a+x)t}}{e^{ft}-1}=\sum_{n=0}^\infty B_{n,\chi}(x) \dfrac{t^n}{n!}
 \]
 and
 \[
B_{n,\chi}(x) =\sum_{k=0}^n \binom{n}{k} B_{k,\chi} x^{n-k},\quad n\geq 0.
\]

One has the following basic properties 
\begin{enumerate}
\item[(i)] (see \cite[page 10]{[Iwasawa]})
 \begin{equation}
\label{eq:1}
B_{n,\chi}(x)=f^{n-1}\sum_{a=1}^{f}\chi(a) B_n(\dfrac{a-f+x}{f}),\quad n\geq0.
\end{equation} 
\item[(ii)] if $\chi\not=\chi_0$, the trivial character, then (see  \cite[Theorem 2]{[Iwasawa]})
\begin{equation}
    \label{eq:generalizedBernoulli}
    \begin{aligned}
    B_{n,\chi}&\not=0 &\text{ for } n\geq 1, \; n\equiv \delta_\chi\pmod2\\
    B_{n,\chi}&=0 &\text{ for } n\geq 1, \;n\not\equiv \delta_\chi\pmod2.
    \end{aligned}
\end{equation}
Here $\delta_\chi=\begin{cases} 0 &\text{ if } \chi(-1)=1\\
1 &\text{ if } \chi(-1)=-1.
\end{cases}$
\end{enumerate}
In this paper, we deal with the quadratic character $\chi_p(a)=\left(\dfrac{a}{p}\right)$, $a\in \Z$. 
The conductor of $\chi_p$ is $p$ and  
\[
\delta_{\chi_p} = \begin{cases} 0 &\text{ if } p\equiv 1\pmod 4\\
1 &\text{ if } p\equiv 3 \pmod 4.
\end{cases}
\]
In the case $p\equiv 3\pmod 4$, i.e. $\delta_{\chi_p}=0$, and $p>3$, one has
\[
B_{1,\chi_p}=-h(-p),
\]
where $h(-p)$ is the class number of the quadratic imaginary field $\Q(\sqrt{-p})$. This follows from the formula (\cite[Theorem 4.9 (i)]{[Washington]})
\[
L(1,\chi_p)=-\dfrac{\pi}{\sqrt{p}}B_{1,\chi_p}
\]
and the Dirichlet class number formula for a quadratic imaginary field (see e.g. \cite[Formula (15), page 49]{[Davenport]})
\[
h(-p)=\dfrac{\sqrt{p}}{\pi}L(1,\chi_p).
\]

\subsection{The trivial roots of Fekete polynomials}

We can see that $x=0$ is a root of $F_p(x)$. Furthermore, we have 
\[ F_p(1)=\sum_{a=1}^{p-1} \left(\frac{a}{p} \right)=0 .\] 
Therefore, $x=1$ is also root of $F_p(x)$. 
 We have the following result, which is \cite[Lemma 4]{[Conrey]}.   

 \begin{prop} 
 \label{prop:trivialroots}
 Let $p$ be a prime number and $F_p (x)$ is the Fekete polynomial
 \begin{enumerate}
     \item[(a)] The number $1$ is a double root of $F_p$ if $p\equiv 1 \pmod 4$ and a simple one if $p\equiv 3\pmod 4$. 
    \item[(b)] The number $-1$ is a simple root of $F_p$ if $p\equiv 1 \pmod 4$ and it is not a root if $p\equiv 3\pmod 4$.
 \end{enumerate}
 \end{prop}

In \cite{[Conrey]} the authors proved this result by considering the roots of certain continuous real function (see \cite[Formula (2.3)]{[Conrey]}). 
We proceed by relating certain sums to generalized Bernoulli numbers and class number.   We need the following lemmas.

\begin{lem}{}
\label{lem:mult}
\begin{enumerate} 
\item[(a)] $\sum\limits_{a=1}^{p-1}\left(\dfrac{a}{p}\right)a =\begin{cases} 0 &\text { if } p\equiv 1\pmod 4\\
 \not=0 &\text { if } p\equiv 3\pmod 4.
\end{cases}$
\item[(b)] $\sum\limits_{a=1}^{p-1}\left(\dfrac{a}{p}\right)a^2\not=0$.
\end{enumerate}
\end{lem}

\begin{proof}

Substituting $n=1$ and $x=0$ in (\ref{eq:1}), one obtains
\[
\begin{aligned}
B_{1,\chi_p}&= \sum_{a=1}^{p-1}\chi_p(a) B_1 \left(\dfrac{a}{p}-1 \right)=\sum_{a=1}^{p-1}\chi_p(a)\left(\dfrac{a}{p}-\dfrac{1}{2}\right)\\
&=\dfrac{1}{p}\sum_{a=1}^{p-1}\chi_p(a)a -\dfrac{1}{2}\sum_{a=1}^{p-1}\chi_p(a)=\dfrac{1}{p}\sum_{a=1}^{p-1}\chi_p(a)a.
\end{aligned}
\]
Hence, by (\ref{eq:generalizedBernoulli}) 
\[
\sum_{a=1}^{p-1}\chi_p(a)a=pB_{1,\chi_p}=\begin{cases} 0 &\text { if } p\equiv 1\pmod 4\\
 \not=0 &\text { if } p\equiv 3\pmod 4.
\end{cases}
\]

Substituting $n=2$ and $x=0$ in (\ref{eq:1}), one obtains
\[
\begin{aligned}
B_{2,\chi_p}&= p\sum_{a=1}^{p-1}\chi_p(a) B_2(\dfrac{a}{p}-1)
=p\sum_{a=1}^{p-1}\chi_p(a)\left(\dfrac{a^2}{p^2}-\dfrac{a}{p}+\dfrac{1}{6}\right)\\
&=\dfrac{1}{p}\sum_{a=1}^{p-1}\chi_p(a)a^2 -\sum_{a=1}^{p-1}\chi_p(a)a.
\end{aligned}
\]
Hence
\[
\sum_{a=1}^{p-1}\chi_p(a)a^2 =p B_{2,\chi_p} + p\sum_{a=1}^{p-1}\chi_p(a)a.
\]
If $p\equiv 1\pmod 4$, then by (\ref{eq:generalizedBernoulli})
\begin{equation}
\label{eq:p=1mod4_moment2}
\sum\limits_{a=1}^{p-1}\chi_p(a)a^2 =p B_{2,\chi_p}\not=0.
\end{equation}
If $p\equiv 3\pmod 4$, then by (\ref{eq:generalizedBernoulli}) \[\sum\limits_{a=1}^{p-1}\chi_p(a)a^2=p\sum\limits_{a=1}^{p-1}\chi_p(a)a=p^2B_{1,\chi_p}\not=0.
\qedhere\]
\end{proof}

\begin{lem} 
\label{lem:Fp(-1)}
One has
\[F_p(-1)=\sum\limits_{a=1}^{p-1}\left(\dfrac{a}{p}\right)(-1)^a =\begin{cases} 0 &\text { if } p\equiv 1\pmod 4\\
 \not=0 &\text { if } p\equiv 3\pmod 4.
\end{cases}
 \]
\end{lem}

\begin{proof}
First, let us consider the case $p \equiv 1 \pmod{4}$. We have 
\begin{align*}
F_p(-1) &=\sum_{a=1}^{p-1} \left(\frac{a}{p} \right)(-1)^a\\
              &=\sum_{a=1}^{\frac{p-1}{2}}  \left [\left(\frac{a}{p} \right)(-1)^a+ \left(\frac{p-a}{p} \right)(-1)^{p-a} \right] \\
              &=\sum_{a=1}^{\frac{p-1}{2}} \left(\frac{a}{p} \right) \left[ (-1)^a+(-1)^{p-a} \right]=0.
\end{align*} 
Note that the in third  equality, we use the fact if that $p \equiv 1 \pmod{4}$ then 
\[ \left(\frac{p-a}{p} \right)=\left(\frac{a}{p} \right) .\] 
Now, let us consider the case $p \equiv 3 \pmod{4}$. We have 
\begin{align*}
F_p(-1) &=\sum_{a=1}^{p-1} \left(\frac{a}{p} \right)(-1)^a\\
              &=\sum_{a=1}^{\frac{p-1}{2}}  \left [\left(\frac{2a}{p} \right)(-1)^{2a}+ \left(\frac{p-2a}{p} \right)(-1)^{p-2a} \right] \\
              &=2\sum_{a=1}^{\frac{p-1}{2}} \left(\frac{2a}{p} \right)=2 \left(\frac{2}{p} \right) \sum_{a=1}^{\frac{p-1}{2}} \left(\frac{a}{p} \right).
\end{align*} 
By \cite[Corollary 3.4]{[Berndt]}, we have 
\[ \sum_{a=1}^{\frac{p-1}{2}} \left(\frac{a}{p} \right)=\left(2- \left(\frac{2}{p} \right) \right) h(-p) .\] 
where $h(-p)$ is the class number of the imaginary quadratic field $\Q(\sqrt{-p})$. Therefore, 
\begin{equation}
 \label{eq:Fp(-1)}   
 F_p(-1)= 2 \left(2 \left(\frac{2}{p} \right)-1 \right) h(-p). 
 \end{equation}
 We conclude that $F_p(-1) \neq 0$ if $p \equiv 3 \pmod{4}$. 
\end{proof} 

\begin{rmk}
The above proof also shows that if $p \equiv 3 \pmod{4}$ then $x=-1$ is not a root of $F_p(x)$ modulo $p$. In fact, we have 
\[  0<\left(2- \left(\frac{2}{p} \right) \right) h(-p)= \sum_{a=1}^{\frac{p-1}{2}} \left(\frac{a}{p} \right)   \leq \frac{p-1}{2} .\] 
Hence, we see that $\sum\limits_{a=1}^{\frac{p-1}{2}} \left(\dfrac{a}{p} \right) \in \{1, 2, \ldots, \frac{p-1}{2} \}$. In particular, we have 
\[ p \nmid \sum_{a=1}^{\frac{p-1}{2}} \left(\frac{a}{p} \right) .\] 
From this, we can see that $p \nmid F_p(-1)$.

\end{rmk}

\begin{lem}
\label{lem:p=1mod4}
If $p\equiv 1\pmod 4$ then 
\begin{equation}
\label{eq:halfsum}
\sum_{a=1}^{\frac{p-1}{2}} \left(\frac{a}{p} \right)=0.
\end{equation}
and
\begin{equation}
    \label{eq:Fp'(-1)}
F'_p(-1)=\sum_{a=1}^{p-1} (-1)^a \left(\frac{a}{p} \right)a 
          = 4 \left(\frac{2}{p} \right) \sum_{a=1}^{\frac{p-1}{2}} \left(\frac{a}{p} \right) a\not=0.
          \end{equation}
\end{lem}
\begin{proof}
We have 
\begin{align*}
 0= \sum_{a=1}^{p-1} \left(\frac{a}{p} \right) &=\sum_{a=1}^{\frac{p-1}{2}}  \left[ \left(\frac{a}{p} \right)+ \left(\frac{p-a}{p} \right) \right]=2\sum_{a=1}^{\frac{p-1}{2}} \left(\frac{a}{p} \right). 
\end{align*} 
Hence 
\[ \sum_{a=1}^{\frac{p-1}{2}} \left(\frac{a}{p} \right)=0.\] 

We have 
\[
\begin{aligned}
F'_p(-1)=\sum_{a=1}^{p-1} (-1)^a \left(\frac{a}{p} \right)a 
              &=\sum_{a=1}^{\frac{p-1}{2}}  \left [ (-1)^{2a} \left(\frac{2a}{p} \right)(2a)+ (-1)^{p-2a}\left(\frac{p-2a}{p} \right)(p-2a) \right] \\
              &=\sum_{a=1}^{\frac{p-1}{2}} \left(\frac{2a}{p} \right)(4a) - p \left(\frac{2}{p} \right) \sum_{a=1}^{\frac{p-1}{2}} \left(\frac{a}{p} \right)= 4 \left(\frac{2}{p} \right) \sum_{a=1}^{\frac{p-1}{2}} \left(\frac{a}{p} \right) a \not=0.
\end{aligned}
\]
Because by \cite[Corollary 13.2]{[Berndt]},  
$ \sum_{a=1}^{\frac{p-1}{2}} \left(\frac{a}{p} \right) a < 0.$ 
\end{proof}

\begin{proof}[Proof of Proposition~\ref{prop:trivialroots}]
(a) By Lemma~\ref{lem:mult} (a), one has 
\[
F_p^\prime(1)=\sum_{a=1}^{p-1}\left(\dfrac{a}{p}\right)a =\begin{cases}
0 &\text{ if } p\equiv 1\pmod 4\\
\not=0 &\text{ if } p\equiv 3\pmod 4.
\end{cases}
\]
So the number $1$ is a simple root of $F_p$ if $p\equiv 3\pmod 4$.

Now we suppose that $p\equiv 1\pmod 4$. 
By Lemma~\ref{lem:mult}, one has
\begin{equation}
\label{eq:Fp''(1)}
\begin{aligned}
F_p''(1)&=\sum_{a=1}^{p-1}\left(\dfrac{a}{p}\right)a(a-1)
=\sum_{a=1}^{p-1}\left(\dfrac{a}{p}\right) a^2 - \sum_{a=1}^{p-1}\left(\dfrac{a}{p}\right) a
=\sum_{a=1}^{p-1}\left(\dfrac{a}{p}\right) a^2 \not=0.
\end{aligned}
\end{equation}
Hence the number $1$ is a double root of $F_p$ if $p\equiv 1\pmod 4$.

(b) Lemma~\ref{lem:Fp(-1)} and Lemma~\ref{lem:p=1mod4} imply that the number $-1$ is not a root of $F_p$ if $p\equiv 3\pmod 4$ and it is a simple root of $F_p$ if $p\equiv 1\pmod 4$.
\end{proof}



\subsection{Condition for the number $-1$ to be a multiple root of $F_p$ modulo $p$}

We discuss the necessary and sufficient condition for $x=-1$ to be a multiple root of $F_p(x)$ modulo $p$ when $p \equiv 1 \pmod{4}$. First, we express this condition in term of the classical Bernoulli numbers.

\begin{prop} \label{prop: B_{(p+3)/2}} 
Let $p \equiv 1 \pmod{4}$ and $p>5$. Then $x=-1$ is a multiple root of $F_p(x)$ modulo $p$ if and only if $p\mid B_{(p+3)/2}$.  
\end{prop} 

\begin{proof}
By (\ref{eq:Fp'(-1)}), $x=-1$ is a multiple root of $F_p(x)$ modulo $p$ if and only if 
\[ p\mid  \sum_{a=1}^{\frac{p-1}{2}} \left(\frac{a}{p} \right) a. \] 
By Euler's criterion, $\left(\frac{a}{p} \right)\equiv a^{\frac{p-1}{2}} \pmod p$, and hence
\[
\sum_{a=1}^{\frac{p-1}{2}} \left(\frac{a}{p} \right) a\equiv \sum_{a=1}^{\frac{p-1}{2}} a^{\frac{p-1}{2}}a\equiv  \sum_{a=1}^{\frac{p-1}{2}} a^{\frac{p+1}{2}}\pmod p.
\]
On the other hand, by \cite[formula (10), page 352]{[Lehmer]},  one has
\[
\sum_{a=1}^{\frac{p-1}{2}} (p-2a)^{2k-1}\equiv (2^{2k}-1) \dfrac{B_{2k}}{2k} \pmod p,
\]
for $k$ with $2k\not\equiv 2\pmod {p-1}$. We choose the integer $k$ such that $2k-1=\dfrac{p+1}{2}$. In this case, $2^{2k}-1 =4\cdot 2^{\frac{p-1}{2}}-1\equiv 4\left(\dfrac{2}{p}\right)-1\pmod p$. Hence one has 

\[
\left(4\left(\dfrac{2}{p}\right)-1\right) \dfrac{B_{\frac{p+3}{2}}}{\frac{p+3}{2}} \equiv (-2)^{\frac{p+1}{2}} \sum_{a=1}^{\frac{p-1}{2}} a^{\frac{p+1}{2}}
\equiv (-2)^{\frac{p+1}{2}}  \sum_{a=1}^{\frac{p-1}{2}} \left(\frac{a}{p} \right) a
\pmod p.\]
Therefore, for $p>5$, $p\mid \sum\limits_{a=1}^{\dfrac{p-1}{2}} \left(\frac{a}{p} \right) a$ if and only if $p\mid B_{\frac{p+3}{2}}$.
\end{proof} 

We have a similar statement using generalized Bernoulli numbers. 
\begin{prop} \label{prop: B_2} 
Let $p \equiv 1 \pmod{4}$ and $p>5$. Then $x=-1$ is a multiple root of $F_p(x)$ modulo $p$ if and only if $p|B_{2, \chi_p}$. 
\end{prop} 

\begin{proof}
By (\ref{eq:Fp'(-1)}), $x=-1$ is a multiple root of $F_p(x)$ modulo $p$ if and only if 
\[ p\mid \sum_{a=1}^{\frac{p-1}{2}} \left(\frac{a}{p} \right) a. \] 
By \cite[Theorem 13.1]{[Berndt]} applied to $\chi_p$, we have  
\[   \sum_{a=1}^{\frac{p-1}{2}} \left(\frac{a}{p} \right) a= -\frac{p \sqrt{p}}{\pi^2} (1-\frac{\chi_p(2)}{4}) L(2, \chi_p) .\] 
Furthermore, by the formula in \cite[Page 12]{[Iwasawa]} we have 
\[ L(2, \chi_p)=\frac{\sqrt{p}}{2} \left(\frac{ 2 \pi}{p} \right)^2 B_{2, \chi_p}.\] 
Combining the above equality, we see that 
\begin{equation} 
\label{eq:halfmoment1}
\sum_{a=1}^{\frac{p-1}{2}} \left(\frac{a}{p} \right) a =-\left(1-\frac{\chi_p(2)}{4} \right) B_{2, \chi_p}.
\end{equation}
Therefore $p\mid \sum\limits_{a=1}^{\dfrac{p-1}{2}} \left(\frac{a}{p} \right) a$ if and only if $p\mid\left(1-\frac{\chi_p(2)}{4} \right) B_{2, \chi_p}$. For $p>5$, this is equivalent to $p|B_{2, \chi_p}$. 

\end{proof}

\begin{rmk}
\label{rmk:congruence_Bernoulli}
We have expressed the necessary and sufficient condition for $x=-1$ to be a multiple root of $F_p(x)$ modulo $p$ using $B_{\frac{p+3}{2}}$ and $B_{2, \chi_p}$. Using the Leopoldt-Kubota $p$-adic $L$-function, we can show that 
\begin{equation} 
\label{eq:congruence}
\frac{B_{2, \chi_p}}{2} \equiv 2 \frac{B_{\frac{p+3}{2}}}{p+3} \pmod{p}. 
\end{equation}

In fact, by \cite[Theorem 5.11]{[Washington]}, for a  Dirichlet character $\chi\not=\chi_0$, $\chi_0$ the trivial character, there exists a $p$-adic analytic function $L_p(s,\chi)$ on a small disk such that
\begin{equation}
\label{eq:p-adicL-function}
L_p(1-n,\chi)=-(1-\chi\omega^{-n}(p)p^{n-1})
\dfrac{B_{n,\chi\omega^{-n}}}{n}, \; n\geq 1.
\end{equation}
Here $\omega$ is the Teichm\"{u}ller character (see \cite[page 51]{[Washington]}). Applying (\ref{eq:p-adicL-function}) for the case $\chi=\omega^{\frac{p+3}{2}}$ and $n=\dfrac{p+3}{2}$, one obtains
\[
L_p(1-\frac{p+3}{2},\omega^{\frac{p+3}{2}})\equiv -\dfrac{B_{\frac{p+3}{2},\chi_0}}{\frac{p+3}{2}} \equiv -\dfrac{B_{\frac{p+3}{2}}}{\frac{p+3}{2}}\pmod p.
\]
On the other hand, applying (\ref{eq:p-adicL-function}) for the case $\chi=\omega^{\frac{p+3}{2}}$ and $n=2$ and noticing that $\omega^{(p-1)/2}=\chi_p$, one obtains
\[
L_p(-1,\omega^{\frac{p+3}{2}})\equiv -\dfrac{B_{2,\omega^{\frac{p-1}{2}}}}{2} \equiv -\dfrac{B_{2,\chi_p}}{2}\pmod p.
\]
By \cite[Corollary 5.13]{[Washington]}, one has
\[
L_p(1-\frac{p+3}{2},\omega^{\frac{p+3}{2}})\equiv L_p(-1,\omega^{\frac{p+3}{2}}) \pmod p.
\]
The desired congruence follows.

From (\ref{eq:congruence}), we can see that the two conditions $p|B_{\frac{p+3}{2}}$ and $p|B_{2, \chi_p}$ are equivalent. 

\end{rmk}

\begin{rmk}
The condition $p|B_{\frac{p+3}{2}}$ implies that $p$ is an irregular prime. In the list of all irregular primes less than $2^{31}$ computed by the authors of \cite{[Irregular]}, $p=89209$ is the only prime number that satisfies the condition $p|B_{\frac{p+3}{2}}$. 

\end{rmk} 

\begin{rmk}
The question of whether $p|B_{2, \chi_p}$ is quite interesting. Let $F=\Q(\sqrt{p})$. By the consequence of the Iwasawa main conjecture proved by Wiles (see \cite[Remark 1.4]{[Kurihara]}) we have 
\begin{equation*} 
\ord_{p} \#H^2(\OO_{F}[1/p], \Z_p(2)) =\ord_{p} (\zeta_{F}(-1))+ \ord_{p} \#H^1(\OO_{F}[1/p], \Z_p(2)). 
\end{equation*} 
By the Quillen-Lichtenbaum's conjecture (now a theorem, see \cite[Theorem 5.6.8]{[Bloch-Kato]}) we have 
\[ H^2(\OO_{F}[1/p], \Z_p(2)) \cong K_2(\OO_{F}) \otimes \Z_p, H^1(\OO_{F}[1/p], \Z_p(2)) \cong K_3(\OO_F) \otimes \Z_p .\] 
Furthermore, the Galois group $\Gal(F/\Q)=\{1, c \}$, with $c$ being the complex conjugation,  acts on all relevant groups. Because $p$ is odd, we have a canonical decomposition  
\[ K_{r}(\OO_{F})\otimes \Z_p = (K_r(\OO_F)^{+} \otimes \Z_p) \bigoplus  (K_r(\OO_F)^{-} \otimes \Z_p) = (K_r(\Z) \otimes \Z_p)  \bigoplus  (K_r(\OO_F)^{-} \otimes \Z_p).\] 

On the $L$-function side we also have $\zeta_{F}(s)=\zeta_{\Q}(s)  L(s, \chi_p)$. In particular, at $s=-1$, we have
\[ \zeta_{F}(-1)=\zeta_{\Q}(-1) L(-1,\chi_p) .\] 
Finally, we have (see \cite[Theorem 1]{[Iwasawa]})
\[ L(-1, \chi_p)=-\frac{B_{2, \chi_p}}{2} .\] 

By the computation in \cite[Table 10.1.1]{[Weibel]}, we have $K_{2}(\Z) \otimes \Z_p= K_{3}(\Z) \otimes \Z_p =K_3(\OO_{F}) \otimes \Z_p=0$ for $p \geq 5$. So in summary, we have
\begin{equation*}
\ord_{p}(|K_2(\OO_{F})|)=\ord_{p}(|K_2(\OO_{F})^{-}|)=\ord_{p}(B_{2, \chi_p}). 
\end{equation*} 
Consequently, if $x=-1$ is a multiple root of $F_p(x)$ modulo $p$, the second $K$-group $K_2(\OO_{F})$ would be non-trivial.
\end{rmk} 

\subsection{{A certain half sum}}
In the proof of Proposition \ref{prop:trivialroots} (b), the sum 
\[ \sum_{a=1}^{p-1} \left(\frac{a}{p} \right) (-1)^a ,\]
appears quite naturally. Following Berndt's article \cite{[Berndt]}, we are interested in the following half-sum 
\[ \sum_{a=1}^{\frac{p-1}{2}} \left(\frac{a}{p} \right) (-1)^a.\]
Through numerical experiments, we found the following result. 
\begin{prop} \label{prop:half_sum}
Let $p$ be an odd prime. Then 
\begin{enumerate} 
\item If $p \equiv \pm 1 \pmod{8}$ then 
\[ \sum_{a=1}^{\frac{p-1}{2}} \left(\frac{a}{p} \right)(-1)^a > 0 .\] 
\item If $p \equiv \pm 5 \pmod{8}$ then 
\[ \sum_{a=1}^{\frac{p-1}{2}} \left(\frac{a}{p} \right)(-1)^a < 0 .\] 
\end{enumerate}

Equivalently, we can summarize both of these statements into a single statement 
\[ \left(\frac{2}{p} \right) \sum_{a=1}^{\frac{p-1}{2}} \left(\frac{a}{p} \right)(-1)^a >0 .\] 
\end{prop} 
\begin{proof}
We first provide a proof for this proposition when $p \equiv 3 \pmod{4}$. First of all, we have the following equality in the case $p \equiv 3 \pmod{4}$.
\begin{align*}
\sum_{a=1}^{p-1} \left(\frac{a}{p} \right)(-1)^a
              &=\sum_{a=1}^{\frac{p-1}{2}}  \left [\left(\frac{a}{p} \right)(-1)^a+ \left(\frac{p-a}{p} \right)(-1)^{p-a} \right] \\
              &=\sum_{a=1}^{\frac{p-1}{2}} \left(\frac{a}{p} \right) \left[ (-1)^a-(-1)^{p-a} \right]\\
              &=2 \sum_{a=1}^{\frac{p-1}{2}} \left(\frac{a}{p} \right)(-1)^a. 
\end{align*} 
Second of all, by (\ref{eq:Fp(-1)}) we have 
\begin{align*}
\sum_{a=1}^{p-1} \left(\frac{a}{p} \right)(-1)^a
              =F_p(-1)=2\left(2\left(\dfrac{2}{p}\right)-1\right)h(-p).
\end{align*} 
Hence, we have 
\begin{equation}
\label{eq:Fp(-1)_halfsum}    
\sum_{a=1}^{\frac{p-1}{2}} \left(\frac{a}{p} \right)(-1)^a = \left(2\left(\dfrac{2}{p}\right)-1\right)h(-p). 
\end{equation} 
Therefore 
\[ \left(\frac{2}{p} \right) \sum_{a=1}^{\frac{p-1}{2}} \left(\frac{a}{p} \right)(-1)^a =\left(2-\left(\dfrac{2}{p}\right)\right)h(-p)>0 .\] 
Let us now consider the case $p \equiv 1 \pmod{4}$. In this case, we have 
\begin{align*} 
\sum_{a=1}^{\frac{p-1}{2}} \left(\frac{a}{p} \right)(-1)^a  &= \sum_{\substack{1 \leq a \leq \frac{p-1}{2} \\ a \equiv 0 \mod{2}}} \left(\frac{a}{p} \right)+\sum_{\substack{1 \leq a \leq \frac{p-1}{2} \\ a \equiv 1 \mod{2}}} \left(\frac{a}{p} \right) \\ 
                    &=\sum_{a=1}^{\frac{p-1}{4}} \left(\frac{2a}{p} \right)- \sum_{\substack{1 \leq a \leq \frac{p-1}{2} \\ a \equiv 1 \mod{2}}} \left(\frac{p-a}{p} \right)\\
                    &=\left(\frac{2}{p} \right) \sum_{a=1}^{\frac{p-1}{4}} \left(\frac{a}{p} \right)- \sum_{\substack{1 \leq a \leq \frac{p-1}{2} \\ a \equiv 1 \mod{2}}} \left(\frac{p-a}{p} \right).
\end{align*} 
Note that if $a$ is odd then $p-a$ is even. Let $2u=p-a$ in the second term. Then $\frac{p+1}{4} \leq u \leq \frac{p-1}{2}$. Because $p \equiv 1 \pmod{4}$, we also have $\frac{p+3}{4} \leq u \leq \frac{p-1}{2}$. Therefore, we have   
\begin{equation*} 
\sum_{\substack{1 \leq a \leq \frac{p-1}{2} \\ a \equiv 1 \mod{2}}} \left(\frac{p-a}{p} \right)=\sum_{u=\frac{p+3}{4}}^{\frac{p-1}{2}} \left(\frac{2u}{p} \right)=\left(\frac{2}{p} \right)\sum_{a=\frac{p+3}{4}}^{\frac{p-1}{2}} \left(\frac{a}{p} \right). 
\end{equation*} 
By  (\ref{eq:halfsum}), we have 
\[  \sum_{a=\frac{p+3}{4}}^{\frac{p-1}{2}} \left(\frac{a}{p} \right)=- \sum_{a=1}^{\frac{p-1}{4}} \left(\frac{a}{p} \right) .\] 
Hence
\[  \sum_{a=1}^{\frac{p-1}{2}} \left(\frac{a}{p} \right)(-1)^a =2 \left(\frac{2}{p} \right) \sum_{a=1}^{\frac{p-1}{4}} \left(\frac{a}{p} \right) .\] 
By \cite[Inequality 1.2]{[Berndt]}, we have 
\[  \sum_{a=1}^{\frac{p-1}{4}} \left(\frac{a}{p} \right)>0 .\] 
Therefore, we have 
\[  \left(\frac{2}{p} \right) \sum_{a=1}^{\frac{p-1}{2}} \left(\frac{a}{p} \right)(-1)^a >0 .\] 
\end{proof} 

\section{The polynomials $f_p(x)$ and $g_p(x)$} 
Let us define
\[
f_p(x)= \begin{cases}
\dfrac{F_p(x)}{x(1-x)} &\text{ if } p\equiv 3 \pmod 4\\
\dfrac{F_p(x)}{x(1-x)^2(x+1)} &\text{ if } p\equiv 1 \pmod 4.
\end{cases}
\]
By the results from the previous section, we know that $f_p(x) \in \Z[x]$. In this section, we investigate some arithmetical properties of $f_p(x)$.

First, we introduce the following general notations.  Let $A$ be a commutative ring with identity. Recall that given a polynomial 
\[
f(x) =a_0+a_1x+\cdots +a_nx^n,
\]
of degree $n$ with coefficients from  $A$, its reciprocal or reflected polynomial, denoted by $f^*$ or $f^R$, is the polynomial
\[
f^*(x)=x^n f \left(\dfrac{1}{x} \right).
\]
The coefficients of $f^*$ are the coefficients of $f$ in reverse order. Polynomial $f$ is called reciprocal or palindromic if $f=f^*$, that means $a_{k}=a_{n-k}$ for all $k$.

We also recall that the Dickson polynomial $D_n(x,a)$ of the first kind of degree $n\geq 1$ in the intermediate $x$ and with parameter $a\in A$ is defined as
\[
D_n(x,a) =\sum_{k=0}^{\lceil n/2 \rceil} \dfrac{n}{n-k} \binom{n-k}{k} (-a)^k x^{n-2k}.
\]
The term $\frac{n}{n-k} \binom{n-k}{k}$ is an integer. Dickson polynomials $D_n(x,a)$ have following two basic properties:
\begin{enumerate}
\item $D_n \left(x+\dfrac{1}{x},a \right)= x^n+\dfrac{a^n}{x^n}$.
\item $D_1(x,a) =x$, $D_2(x,a)= x^2-2a$, and
\[
D_n(x,a)= xD_{n-1}(x,a) - a D_{n-2}(x).
\] 
\end{enumerate}
For simplicity we will write $D_n(x)$ for $D_n(x,1)$ so that $D_n \left(x+\dfrac{1}{x} \right)= x^n+\dfrac{1}{x^n}$.

Now suppose that  $f(x) =a_0+a_1x+\cdots +a_nx^n\in A[x]$ is a reciprocal polynomial of even degree $n$. Write $n=2s$ and set 
\[
g(x) =\sum_{k=0}^{s-1} a_k D_{n-k}(x) +a_s \in A[x].
\]
Then $f(x) =x^s g(x+\dfrac{1}{x})$.

We have the following proposition. 
\begin{prop}
$f_p(x)$ is a reciprocal polynomial of even degree. 

\end{prop} 

\begin{proof}
Let us consider the case $p \equiv 3 \pmod{4}$. In this case, we have 
\[f_p(x)=\dfrac{F_p(x)}{x(1-x)}.\] 
Let us first consider the Fekete polynomial $F_p(x)$, we have 
\begin{align*}
x^p F_p \left(\frac{1}{x} \right) &= x^p \sum_{a=1}^{p-1} \left(\frac{a}{p} \right) \left(\frac{1}{x} \right)^{a}=\sum_{a=1}^{p-1}  \left(\frac{a}{p} \right) x^{p-a}\\
&=\sum_{u=1}^{p-1} \left(\frac{p-u}{p} \right)x^u= - \sum_{u=1}^{p-1} \left(\frac{u}{p} \right)x^u=-F_p(x).       \end{align*} 
We then have 
\begin{align*}
x^{p-3} f_p \left(\frac{1}{x} \right) &= x^{p-3} \dfrac{F_p \left(\frac{1}{x} \right)}{\dfrac{1}{x} \left(1-\dfrac{1}{x} \right)} = \dfrac{-x^p F_p(1/x)}{x(1-x)}\\
&=\dfrac{F_p(x)}{x(1-x)}=f_p(x).                            \end{align*} 
Note that the degree of $f_p$ is $\dfrac{p-3}{2}$ which is even. Therefore $f_p(x)$ is a reciprocal polynomial of even degree. 

Next, let us consider the case $p \equiv 1 \pmod{4}$. As the previous case, let us consider 
\begin{align*}
x^p F_p \left(\frac{1}{x} \right) &= x^p \sum_{a=1}^{p-1} \left(\frac{a}{p} \right) \left(\frac{1}{x} \right)^a =\sum_{a=1}^{p-1}  \left(\frac{a}{p} \right) x^{p-a} \\
&=\sum_{u=1}^{p-1} \left(\frac{p-u}{p} \right)x^u=  \sum_{u=1}^{p-1} \left(\frac{u}{p} \right)x^u=F_p(x).                                                               
\end{align*} 

We then have 

\begin{align*}
x^{p-5} f_p \left(\frac{1}{x} \right) &= x^{p-5} \dfrac{F_p \left(\frac{1}{x} \right)}{\frac{1}{x} \left(\frac{1}{x}+1 \right) \left(\frac{1}{x}-1 \right)^2} \\
&= \dfrac{x^p F_p(1/x)}{x(1+x)(x-1)^2}=\dfrac{F_p(x)}{x(1+x)(x-1)^2}=f_p(x).                                                            
\end{align*}
Note that the degree of $f_p(x)$ is $\frac{p-5}{2}$ which is even. We conclude that $f_p(x)$ is a reciprocal polynomial of even degree.
\end{proof}
It is natural to define the following related polynomial. 

\begin{definition}
Let $g_p(x) \in \Z[x]$ be the polynomial such that 
\[ f_p(x)=x^{\frac{\deg(f_p)}{2}} g_p \left(x+\frac{1}{x} \right) .\]
We will call $g_p(x)$ the reduced Fekete polynomial associated with $p$.
\end{definition} 

We provide the explicit formulas for $f_p(x)$ and $g_p(x)$ for  $p \leq 23$. Here are some explicit formulas for $f_p$. 
\[ f_{7}=x^4+2x^3+x^2+2x+1 .\] 
\[ f_{11}=x^8+x^6+2x^5+3x^4+2x^3+x^2+1 \] 
\[ f_{13}(x)=x^8 + 2 x^6 + 2 x^5 + 3 x^4 + 2 x^3 + 2 x^2 + 1. \] 
\[ f_{17}=x^{12} + 2x^{11} + 2x^{10} + 4x^9 + 3x^8 + 4x^7 + 2x^6 + 4x^5 + 3x^4 + 4x^3 + 2x^2 + 2x + 1 .\] 
\[ f_{19}=x^{16}-x^{14}+x^{12}+2x^{11}+3x^{10}+2x^{9}+3x^8+2x^7+3x^6+2x^5+x^4-x^2+1 .\] 
\begin{multline*}
f_{23}(x)=x^{20}+2x^{19}+3x^{18}+4x^{17}+3x^{16}+4x^{15} +3x^{14}+ 4x^{13}+5x^{12}+4x^{11} \\+3x^{10}+4x^{9}+5x^8+4x^7+3x^6+4x^5+3x^4+4x^3+3x^2+2x+1
\end{multline*}

Here are some explicit formulas for $g_p(u)$.
\[ g_7(u)=u^2 + 2u - 1 .\] 
\[ g_{11}(u)=u^4 - 3u^2 + 2u + 3. \] 
\[g_{13}(u)=u^4 - 2u^2 + 2u + 1.\]
\[ g_{17}(u)=u^6 + 2u^5 - 4u^4 - 6u^3 + 4u^2 + 2u - 2.\]
\[g_{19}(u)=u^8 - 9u^6 + 27u^4 + 2u^3 - 26u^2 - 4u+ 3.\]
\[ g_{23}(u)=u^{10} + 2u^9 - 7u^8 - 14u^7 + 14u^6 + 30u^5 - 5u^4 - 20u^3 - 3u^2 + 2u - 3.\]

It turns out that the special values of $g_p(x)$ contains lots of arithmetic information. We will demonstrate this observation by several propositions. For $p=3$ or $p=5$, $g_p(x)=1$ so these are the trivial cases. From now on, we assume that $p \geq 7$. Let us recall that for a quadratic extension $\Q(\sqrt{s})$ ($s \in \Q^{\times}$) of $\Q$, we denote by $h(s)$ its class number.

\begin{prop}
If $p \equiv 3 \pmod{4}$ then 
\[ g_p(2)=f_p(1)=p h(-p).\]

If $p \equiv 1 \pmod{4}$ then 
\[ g_p(2)= f_p(1)=\frac{p B_{2,\chi_p}}{4}.\]

\end{prop}  

\begin{proof}
Let us first consider the case $p \equiv 3 \pmod{4}$. In this case $\dfrac{p-3}{2}$ is even, we have 
\[ g_p(2)=f_p(1) .\] 
We also have 
\[ xf_p(x)=\dfrac{F_p(x)}{1-x}.\] 
Taking the limit when $x \to 1$, we have 
\[ f_p(1)=F_p^{'}(1)=-\sum_{r=1}^{p-1} \left(\frac{r}{p} \right) r. \] 
The right hand side is a classical sum. More precisely we have the following class number formula (see \cite[Equation 3]{[ClassNumber]}) 
\[ \sum_{r=1}^{p-1} \left(\frac{r}{p} \right) r= -p h(-p).\] 
Hence, we see that 
\[ g_p(2)=f_p(1)=p h(-p).\]
Next, let us consider the case $p \equiv 1 \pmod{4}$. As above, we have 
\[ g_p(2)=f_p(1) .\] 
Let us now compute $f_p(1)$. We have 
\[ x(x+1)f_p(x)=\dfrac{F_p(x)}{(x-1)^2}.\] 
Taking the limit of both sides when $x \to 1$ and using (\ref{eq:p=1mod4_moment2}) and (\ref{eq:Fp''(1)}) we have 
\[ 2f_p(1)=\dfrac{F_{p}^{''}(1)}{2}=\dfrac{pB_{2, \chi_p}}{2}.\] 
Therefore 
\[f_p(1)=\frac{p B_{2,\chi_p}}{4}. 
\qedhere
\]
\end{proof}

\begin{prop} 
If $p \equiv 3 \pmod{4}$ then 
\[  g_p(-2)=f_p(-1)=-\left(2 \left(\frac{2}{p} \right)-1 \right)h(-p). \] 
On the other hand, if $p \equiv 1 \pmod{4}$ then
\[ f_p(-1)=g_p(-2)=-\frac{1}{4} \left( 4 \left(\frac{2}{p} \right)-1 \right) B_{2, \chi_p}.\] 
Here $B_{2, \chi_p}$ is the generalized Bernoulli number associated with the character $\chi_p$ that was introduced in the second section. 
\end{prop}  

\begin{proof}
Let us consider the case $p \equiv 3 \pmod{4}$. 
Using (\ref{eq:Fp(-1)}) we have 
\[  g_p(-2)=f_p(-1)=-\dfrac{F_p(-1)}{2}=-\left(2 \left(\frac{2}{p} \right)-1 \right)h(-p).\] 
Next, let us consider the case $p \equiv 1 \pmod{4}$. In this case $\dfrac{p-5}{2}$ is even and we have 
\[ g_p(-2)=f_p(-1).\] 
By definition, we have 
\[ x(x-1)^2 f_p(x)=\dfrac{F_p(x)}{x+1} .\] 
Taking the limit when $x \to -1$ we get 
\[ -4f_p(-1)= (-1)(-2)^2 f_p(-1)=F_{p}^{'}(-1).\]
Thus, by (\ref{eq:Fp'(-1)}) and (\ref{eq:halfmoment1})
\[
\begin{aligned}
f_p(-1)&=-\dfrac{1}{4}F'_p(-1)=\left(\frac{2}{p} \right) \sum_{a=1}^{\frac{p-1}{2}} \left(\frac{a}{p} \right) a\\
&=-\left(\frac{2}{p} \right)\left(1-\frac{\chi_p(2)}{4} \right) B_{2, \chi_p}=-\frac{1}{4} \left( 4 \left(\frac{2}{p} \right)-1 \right) B_{2, \chi_p}.
\qedhere
\end{aligned}
\]
\end{proof} 

Next, we compute the values of $g_p(x)$ at $0$. 

\begin{prop}
Let $p \equiv 3 \pmod{4}$.  Then 
\[ g_p(0)=g_p(-2)=-\left(2 \left(\frac{2}{p} \right)-1 \right)h(-p) .\] 
\end{prop} 
\begin{proof}

The statement for $g_p(-2)$ is a direct consequence of the previous proposition. Let us focus on the case $g_p(0)$. First, we recall that  
\[ f_p(x)=x^{\frac{p-3}{2}} g_p \left(x+\frac{1}{x} \right) .\] Plugging $x=i=\sqrt{-1}$ into this equation gives 
\[ f_p(i)=i^{\frac{p-3}{2}} g_p(0)=-\left(\frac{2}{p} \right) g_p(0) .\] 
Now, let us compute $f_p(i)$. We have 
\[ f_p(i)=-\frac{F_p(i)}{i(i-1)}=\frac{F_p(i)}{i+1} .\] 
By the above equality we have 
\[ g_p(0)=-\frac{\left(\frac{2}{p} \right)F_p(i)}{i+1} .\]
We then have 
\begin{align*}
F_p(i) &=\sum_{i=1}^{p-1} \left(\frac{a}{p} \right)(-i)^a 
             =\sum_{a=1}^{\frac{p-1}{2}} \left[\left(\frac{2a}{p} \right) i^{2a} + \left(\frac{p-2a}{p} \right) i^{p-2a} \right] \\
             &=(1+i)\sum_{a=1}^{\frac{p-1}{2}} \left(\frac{2a}{p} \right)i^{2a}=(1+i) \left(\frac{2}{p} \right) \sum_{a=1}^{\frac{p-1}{2}} \left(\frac{a}{p} \right)(-1)^a\\
             &=(1+i)\left(\frac{2}{p} \right)\left(2\left(\frac{2}{p} \right)-1\right)h(-p) \quad (\text{by } (\ref{eq:Fp(-1)_halfsum})).\\
      \end{align*} 
Therefore, we have 
\[ g_p(0)=-\left(2 \left(\frac{2}{p} \right)-1 \right)h(-p).
\qedhere\]
\end{proof} 
By the same method, we can also compute $g_p(-1)$. Here we note that if $\zeta_{3}$ is the cubic root of $1$, namely $\zeta_{3}=\exp(\frac{2 \pi i}{3})=\dfrac{-1+\sqrt{-3}}{2}$ then 
\[ \zeta_{3}+\frac{1}{\zeta_{3}}=-1 .\] 
We then have 
\begin{equation}  \label{eq:zeta_3}
g_p(-1)=g_p \left(\zeta_{3}+\frac{1}{\zeta_3} \right)=-\frac{F_p(\zeta_{3})}{\zeta_{3}^{\frac{p-1}{2}}(\zeta_3-1)}.
\end{equation} 
Now, let us compute $F_p(\zeta_3)$. By definition 
\begin{equation*} 
F_{p}(\zeta_3)=\sum_{a=1}^{p-1} \left(\frac{a}{p} \right) \zeta_{3}^{a}. 
\end{equation*} 
We can break this sum into three sums according to $a \pmod{3}$. More precisely 
\begin{align*}
\sum_{a=1}^{p-1} \left(\frac{a}{p} \right) \zeta_{3}^{a} &= \sum_{0<3a <p} \left(\frac{3a}{p} \right) \zeta_{3}^{3a} +  \sum_{0<3a+1 <p} \left(\frac{3a+1}{p} \right) \zeta_{3}^{3a+1}+\sum_{0<3a+2 <p} \left(\frac{3a+2}{p} \right) \zeta_{3}^{3a+2}  \\
      &=\sum_{0<a \leq \frac{p-1}{3}} \left(\frac{3a}{p} \right) + \zeta_{3} \sum_{0<3a+1 <p} \left(\frac{3a+1}{p} \right) +\zeta_{3}^2 \sum_{0<3a+2 < p} \left(\frac{3a+2}{p} \right).  
\end{align*} 
First, let us consider the case $p \equiv 1 \pmod{3}$. In this case, the two sets $\{3a \}$ and $\{p-3a-1\}$ are the same. Similarly, the two sets $\{3a+2 \}$ and $\{p-3a-2 \}$ are the same. Additionally, when $p \equiv 3 \pmod{4}$, we have  
\begin{align*}
 \sum_{0<3a+1 <p} \left(\frac{3a+1}{p} \right) &=- \sum_{0<3a+1 < p} \left(\frac{p-3a-1}{p} \right)= -\sum_{0<3a<p} \left(\frac{3a}{p} \right).\\
\end{align*}
Simillarly 
\begin{align*}
 \sum_{0<3a+2 <p} \left(\frac{3a+2}{p} \right) &=- \sum_{0<3a+2 < p} \left(\frac{p-3a-2}{p} \right)= -\sum_{0<3a+2<p} \left(\frac{3a+2}{p} \right).\\
\end{align*}
Consequently, 
\[  \sum_{0<3a+2 <p} \left(\frac{3a+2}{p} \right) =0.\]
In summary, we have 
\[ F_{p}(\zeta_3)=\sum_{a=1}^{p-1} \left(\frac{a}{p} \right) \zeta_{3}^{a}=(1-\zeta_{3}) \sum_{0<3a<p} \left(\frac{3a}{p} \right)=(1-\zeta_3) \left(\frac{3}{p} \right) \sum_{0<a<p/3} \left(\frac{a}{p} \right) .\] 
By \cite[Corollary 4.3]{[Berndt]} we have 
\[  \sum_{0<a<p/3} \left(\frac{a}{p} \right)=\frac{1}{2} \left(3-\left(\frac{3}{p} \right)  \right) h(-p) .\] 
We then have 
\[ F_p(\zeta_3)=\frac{1}{2} (1-\zeta_3) \left(3\left(\frac{3}{p}\right)- 1 \right) h(-p).\]
From Equation \ref{eq:zeta_3} and the fact that $p \equiv 1 \pmod{3}$, we have 
\[ g_p(-1)= -\frac{F_p(\zeta_{3})}{\zeta_{3}^{\frac{p-1}{2}}(\zeta_3-1)}=\frac{1}{2} \left(3\left(\frac{3}{p}\right)-1  \right) h(-p).\]

Now, let us consider the case $p \equiv 2 \pmod{3}$. By the same method, we can see that in this case 

\begin{align*}
 \sum_{0<3a+2 <p} \left(\frac{3a+2}{p} \right) &=- \sum_{0<3a+2 < p} \left(\frac{p-3a-2}{p} \right)
                  = -\sum_{0<3a<p} \left(\frac{3a}{p} \right).\\
\end{align*}
Additionally 
\[  \sum_{0<3a+1 <p} \left(\frac{3a+1}{p} \right) =0.\]
Therefore 
\[ F_p(\zeta_3)=(1-\zeta_3^2) \left(\frac{3}{p} \right) \sum_{0<a<p/3} \left(\frac{a}{p} \right) .\] 
By \cite[Corollary 4.3]{[Berndt]} we have 
\[  \sum_{0<a<p/3} \left(\frac{a}{p} \right)=\frac{1}{2} \left(3-\left(\frac{3}{p} \right)  \right) h(-p) .\] 
We then have 
\[ F_p(\zeta_3)=\frac{1}{2} (1-\zeta_3^2) \left(3\left(\frac{3}{p}\right)- 1 \right) h(-p).\]
We note that when $p \equiv 2 \pmod{3}$ we have $\zeta_{3}^{\frac{p-1}{2}}=\zeta_{3}^2$. Therefore 
\[ \zeta_{3}^{\frac{p-1}{2}} (\zeta_3-1)=\zeta_3^2(\zeta_3-1)=1-\zeta_{3}^2 .\] 
We conclude that in this case we have 
\[ g_p(-1)= -\frac{F_p(\zeta_{3})}{\zeta_{3}^{\frac{p-1}{2}}(\zeta_3-1)}=-\frac{1}{2} \left(3\left(\frac{3}{p}\right)-1  \right) h(-p) .\]
In summary we have the following proposition. 
\begin{prop}
Let $p \equiv 3 \pmod{4}$ then 
\[ g_p(-1) =\begin{cases} \frac{1}{2} \left(3\left(\frac{3}{p}\right)-1  \right) h(-p) &\text { if } p\equiv 1\pmod 3\\
 -\frac{1}{2} \left(3\left(\frac{3}{p}\right)-1  \right) h(-p) &\text { if } p\equiv  2 \pmod 3.
\end{cases}\]
In other words, we have 
\[ g_p(-1)=-\frac{1}{2} \left(\left(\frac{p}{3} \right)+3 \right) h(-p).\]
\end{prop}  
Let us go further to compute $g_p(1)$. To do so, we use the $6$-primitive root of unity, namely $x=\zeta_{6}=\exp(\frac{2 \pi i}{6})$. We first note that  \[1=\zeta_6+\frac{1}{\zeta_6}.\]
As before, we have 
\[ g_{p}(1)= -\frac{F_p(\zeta_6)}{\zeta_6^{\frac{p-1}{2}}(\zeta_6-1)} .\]
We have 
\[ F_p(\zeta_6)= \sum_{i=0}^{5} \zeta_{6}^i \left[ \sum_{0<6n+i<p} \left(\frac{6n+i}{p} \right)\right] .\]
To compute this sum, we use the same technique as before. First, we note that when $p \equiv 1 \pmod{3}$, the following sets are the same 
\[ \{6n+1 \}=\{p-6n \}, \{6n+4\}=\{p-(6n+3) \}, \{6n+5 \}=\{p-(6n+2) \} .\] 
From this observation we have the following identities 
\[  \sum_{0<6n+1<p} \left(\frac{6n+1}{p} \right)=-  \sum_{0<6n<p} \left(\frac{6n}{p} \right) ,\] 
\[  \sum_{0<6n+4<p} \left(\frac{6n+4}{p} \right)=-  \sum_{0<6n+3<p} \left(\frac{6n+3}{p} \right) ,\] 
\[  \sum_{0<6n+5<p} \left(\frac{6n+5}{p} \right)=-  \sum_{0<6n+2<p} \left(\frac{6n+2}{p} \right).\] 
Therefore, we see that 
\[ F_p(\zeta_6)=(1-\zeta_6)  \sum_{0<6n<p} \left(\frac{6n}{p} \right) +(\zeta_6^2-\zeta_6^5)  \sum_{0<6n+2<p} \left(\frac{6n+2}{p} \right)+(\zeta_6^3-\zeta_6^4)  \sum_{0<6n+3<p} \left(\frac{6n+3}{p} \right) .\]
Now, let us simplify the second and the third sums. We have 
\begin{align*}
\sum_{0<6n+2<p} \left(\frac{6n+2}{p} \right) &=  \left(\frac{2}{p} \right) \sum_{0<3n+1<p/2} \left(\frac{3n+1}{p} \right)\\
       &=-\left(\frac{2}{p} \right) \sum_{0<3n+1<p/2}  \left(\frac{p-(3n+1)}{p} \right).
\end{align*}
Let $3u=p-(3n+1)$. Then $u \in \Z$ and $p/6<u<p/3$. Therefore, the above identity can be rewritten as 
\[ \sum_{0<6n+2<p} \left(\frac{6n+2}{p} \right) =-\left(\frac{6}{p} \right) \sum_{p/6<u<p/3} \left(\frac{u}{p} \right) =-\left(\frac{6}{p} \right) S_{26}.\]
Here we use the notations $S_{ij}$ as introduced in \cite[Page 265]{[Berndt]}.  By a similar computation, we can see that 
\[ \sum_{0<6n+3<p} \left(\frac{6n+3}{p} \right)=-\left(\frac{6}{p} \right) \sum_{p/3<u<p/2} \left(\frac{u}{p} \right)=-\left(\frac{6}{p} \right) S_{36} .\] Finally, we have 
\[ \sum_{0<6n<p} \left(\frac{6n+2}{p} \right) =\left(\frac{6}{p} \right) S_{16} .\] 
In summary, we have 
\begin{align*}
F_{p}(\zeta_6) &=\left(\frac{6}{p} \right) \left[ (1-\zeta_6)S_{16}-(\zeta_6^2-\zeta_6^5) S_{26}-(\zeta_6^3-\zeta_6^4) S_{36} \right]\\
                       &=-\zeta_6^2 \left(\frac{6}{p} \right) \left[ S_{16}+2 S_{26}+S_{36} \right].
\end{align*} 
Here we use the following identities
\[ 1-\zeta_6=-\zeta_6^2, \zeta_6^2-\zeta_6^5=2 \zeta_{6}^2, \zeta_6^3-\zeta_6^4=\zeta_6^2 .\] 
By \cite[Theorem 6.1]{[Berndt]}, we have 
\[ S_{16}=\frac{h(-p)}{2} \left[1+\left(\frac{2}{p} \right)+\left(\frac{3}{p} \right)-\left(\frac{6}{p} \right) \right],\] 

\[ S_{26}=\frac{h(-p)}{2} \left[2-\left(\frac{2}{p} \right)-2\left(\frac{3}{p} \right)+\left(\frac{6}{p} \right) \right],\] 

\[ S_{36}=\frac{h(-p)}{2} \left[1-2\left(\frac{2}{p} \right)+\left(\frac{3}{p} \right) \right],\] 
By some simple algebraic calculations, we have 
\[ S_{16}+2S_{26}+S_{36}= 6-3 \left(\frac{2}{p} \right) -2 \left(\frac{3}{p} \right)+\left(\frac{6}{p} \right) .\]
Using these equations, we conclude that 
\[ F_{p}(\zeta_6)= -\zeta_6^{2} \times \frac{h(-p)}{2} \left(\frac{6}{p} \right) \left[ 6-3 \left(\frac{2}{p} \right) -2 \left(\frac{3}{p} \right)+\left(\frac{6}{p} \right)\right].\] 
Note that when $p \equiv 1 \pmod{3}$ we have $\zeta_{6}^{\frac{p-1}{2}}=-1$. Additionally, we note that $\zeta_6-1=\zeta_6^2$. Consequently, we have 
\[ g_p(1)=-\frac{h(-p)}{2} \left(\frac{6}{p} \right) \left[6-3 \left(\frac{2}{p} \right) -2 \left(\frac{3}{p} \right)+\left(\frac{6}{p} \right) \right] .\] 
Now, let us consider the case $p \equiv 2 \pmod{3}$. In this case, we observe that the following sets are the same 
\[ \{6n+5 \}=\{p-6n \}, \{6n+4\}=\{p-(6n+1) \}, \{6n+2 \}=\{p-(6n+3) \}  .\] 
Therefore 
\[ F_p(\zeta_6)=(1-\zeta_6^5)  \sum_{0<6n<p} \left(\frac{6n}{p} \right) +(\zeta_6^4-\zeta)  \sum_{0<6n+4<p} \left(\frac{6n+2}{p} \right)+(\zeta_6^3-\zeta_6^2)  \sum_{0<6n+3<p} \left(\frac{6n+3}{p} \right) .\]
By the same arguments in in the case $p \equiv 1 \pmod{3}$, we have
\[ \sum_{0<6n+2<p} \left(\frac{6n+2}{p} \right) =-\left(\frac{6}{p} \right) \sum_{p/6<u<p/3} \left(\frac{u}{p} \right) =-\left(\frac{6}{p} \right) S_{26},\]
and 
\[ \sum_{0<6n+3<p} \left(\frac{6n+3}{p} \right)=-\left(\frac{6}{p} \right) \sum_{p/3<u<p/2} \left(\frac{u}{p} \right)=-\left(\frac{6}{p} \right) S_{36} .\]  
By \cite[Theorem 6.1]{[Berndt]}, we have 
\begin{align*} 
F_{p}(\zeta_6) &=\left(\frac{6}{p} \right) \left[ (1-\zeta_6^5)S_{16}-(\zeta_6^4-\zeta_6) S_{26}-(\zeta_6^3-\zeta_6^2) S_{36} \right] \\
                        &=\zeta_6 \left(\frac{6}{p} \right) \left[ S_{16}+2 S_{26}+S_{36} \right]\\
                        &=\zeta_6 \frac{h(-p)}{2} \left(\frac{6}{p} \right) \left[6-3 \left(\frac{2}{p} \right) -2 \left(\frac{3}{p} \right)+\left(\frac{6}{p} \right)\right].
\end{align*} 
When $p \equiv 2 \pmod{3}$, we also have 
\[ \zeta_{6}^{\frac{p-1}{2}}(\zeta_6-1)=\zeta_6 .\] 
Hence 
\[g_p(1)=-\frac{F_p(\zeta_6)}{\zeta_{6}^{\frac{p-1}{2}}(\zeta_6-1)}=-\frac{h(-p)}{2} \left(\frac{6}{p} \right)  \left[6-3 \left(\frac{2}{p} \right) -2 \left(\frac{3}{p} \right)+\left(\frac{6}{p} \right)\right].\]
In summary, we have just showed that. 
\begin{prop}
Let $p \equiv 3 \pmod{4}$. Then 
\[ g_p(1)=-\frac{h(-p)}{2} \left(\frac{6}{p} \right)  \left[6-3 \left(\frac{2}{p} \right) -2 \left(\frac{3}{p} \right)+\left(\frac{6}{p} \right)\right].\]
\end{prop} 
Next, let us compute $g_p(0)$, $g_p(1)$ and $g_p(-1)$ when $p \equiv 1 \pmod{4}$. First, let us compute $g_p(-1)$. We have 
\[ g_p(0)=\frac{F_p(i)}{i^{\frac{p-3}{2}} (1+i)(i-1)^2} =\frac{F_p(i)}{-2(i+1)i^{\frac{p-1}{2}}} .\]
We have 
\[ F_p(i)=\sum_{i=1}^{p-1} \left(\frac{n}{p} \right)i^{n} =\sum_{i=0}^{3} i^a \left[\sum_{0<4n+i<p} \left(\frac{4n+i}{p} \right)  \right].\] 
The following sets are the same 
\[ \{4n \}=\{p-(4n+1)\}, \{4a+3\}=\{p-(4n+3) \} .\] 
Therefore we have 
\[ \sum_{0<4n+1<p} \left(\frac{4a+1}{p} \right)=\sum_{0<4n+1<p} \left(\frac{p-(4n+1)}{p} \right)=\sum_{0<4n<p} \left(\frac{4n}{p} \right) =S_{14} .\]
Similarly, we have 
 \[ \sum_{0<4n+3<p} \left(\frac{4n+3}{p} \right)=\sum_{0<4n+3<p} \left(\frac{p-(4n+3)}{p} \right)=\sum_{0<4n+2<p} \left(\frac{4n+2}{p} \right) .\]
 We have 
 \begin{align*}
\sum_{0<4n+2<p} \left(\frac{4n+2}{p} \right) &=  \left(\frac{2}{p} \right) \sum_{0<2n+1<p/2} \left(\frac{2n+1}{p} \right)\\
       &=\left(\frac{2}{p} \right) \sum_{0<2n+1<p/2}  \left(\frac{p-(2n+1)}{p} \right)\\
       &=\left(\frac{2}{p} \right) \sum_{p/4<u<p/2}  \left(\frac{2u}{p} \right)=\sum_{p/4<u<p/2} \left(\frac{u}{p} \right).
\end{align*}
By (\ref{eq:halfsum}), we have 
\[ \sum_{p/4<u<p/2} \left(\frac{u}{p} \right)=- \sum_{0<u<p/4} \left(\frac{u}{p} \right)=-S_{14} .\] 
We then deduce that 
\[ F_p(\zeta_4)=(1+i)S_{14}-(i^2+i^3)S_{14}=2 (i+1)S_{14} .\]
Hence 
\[ g_p(-1)=\frac{F_p(\zeta_4)}{-2(i+1)i^{\frac{p-1}{2}}}=\frac{2(i+1)S_{14}}{-2(i+1) i^{\frac{p-1}{2}}}=-\frac{S_{14}}{(-1)^{\frac{p-1}{4}}}=-\left(\frac{2}{p} \right) S_{14} .\]
By \cite[Corollary 3.9]{[Berndt]}, we have $S_{14}=\frac{1}{2} h(-4p)$. Therefore, we have 
\begin{prop}
Let $p \equiv 1 \pmod{4}$, then
\[g_p(0)=-\frac{1}{2} \left(\frac{2}{p} \right) h(-4p).\] 
\end{prop} 
Next, we will use $\zeta_{6}$ to compute this $g_p(1)$. We have 
\[ g_p(1)=\frac{F_p(\zeta_6)}{\zeta_6^{\frac{p-3}{2}} (1+\zeta_6)(\zeta_6-1)^2}= -\frac{F_p(\zeta_6)}{\zeta_6^{\frac{p-1}{2}}(1+\zeta_6)}.\]
Let us consider the case $p \equiv 1 \pmod{3}$. By the same argument as in the case $p \equiv 3 \pmod{4}$ we have
\begin{align*}
F_{p}(\zeta_6) &=\left(\frac{6}{p} \right) \left[ (1+ \zeta_6)S_{16}+(\zeta_6^2+\zeta_6^5) S_{26}+(\zeta_6^3+\zeta_6^4) S_{36} \right]\\
                       &=(1+\zeta_6) \left(\frac{6}{p} \right) \left[ S_{16}-S_{36} \right].
\end{align*} 
We then have 
\[ g_p(1)=-\left(\frac{6}{p} \right) \frac{S_{16}-S_{36}}{\zeta_{6}^{\frac{p-1}{2}}} =-\left(\frac{6}{p} \right)[S_{16}-S_{36}].\]
By \cite[Theorem 6.1]{[Berndt]} we have 
\[ S_{16}=\frac{1}{2} \left(1+\left(\frac{2}{p} \right) \right) h(-3p) , \] 
and 
\[ S_{36}=- \frac{1}{2}h(-3p) .\] 
Hence 
\begin{align*}
g_p(1) &=-\frac{1}{2} \left(\frac{6}{p} \right) \left(2+\left(\frac{2}{p} \right)  \right) h(-3p) \\
          &=- \frac{1}{2} \left(\frac{p}{3} \right) \left(\frac{6}{p} \right) \left(2+\left(\frac{2}{p} \right)  \right) h(-3p)\\
          &=- \frac{1}{2} \left(\frac{p}{3} \right) \left(\frac{3}{p} \right) \left(\frac{2}{p} \right) \left(2+\left(\frac{2}{p} \right)  \right) h(-3p)\\
          &=-\frac{1}{2}  \left(2 \left(\frac{2}{p} \right) +1  \right) h(-3p)
\end{align*} 
Note that, in the third equality, we use the quadratic reciprocity law 
\[  \left(\frac{p}{3} \right) \left(\frac{3}{p} \right) =1, \] 
as in our case $p \equiv 1\pmod{4}.$ Now, let us consider the case $p \equiv 2 \pmod{3}$. Then we have 
\begin{align*} 
F_{p}(\zeta_6) &=\left(\frac{6}{p} \right) \left[ (1+\zeta_6^5)S_{16}+(\zeta_6^4+\zeta_6) S_{26}+(\zeta_6^3+\zeta_6^2) S_{36} \right] \\
                        &=(1-\zeta_6^2) \left(\frac{6}{p} \right) \left[ S_{16}-S_{36} \right].
\end{align*} 
Therefore 
\[ g_p(1)=-\frac{F_p(\zeta_6)}{\zeta_6^{\frac{p-1}{2}}(1+\zeta_6)}=\left(\frac{6}{p} \right)[S_{16}-S_{36}]=\frac{1}{2} \left(\frac{6}{p} \right) \left(2+\left(\frac{2}{p} \right)  \right) h(-3p) .\]
By the same calculation as in the case $p \equiv 1 \pmod{3}$, the above sum can be simplify to 
\[ g_p(1)=-\frac{1}{2}  \left(2 \left(\frac{2}{p} \right) +1  \right) h(-3p) .\]
We obtain the following proposition. 
\begin{prop} 
Let $p \equiv 1 \pmod{4}$, then 
\[ g_p(1)=-\frac{1}{2}  \left(2 \left(\frac{2}{p} \right) +1  \right) h(-3p) .\]
\end{prop} 
Finally, let us compute $g_p(-1)$. We have 
\[ g_p(-1)=\frac{F_p(\zeta_6)}{\zeta_3^{\frac{p-3}{2}} (1+\zeta_3)(\zeta_3-1)^2}=\frac{F_p(\zeta_{3})}{3 \zeta_{3}^{\frac{p-3}{2}}}.\]
Let us consider the case $p \equiv 1 \pmod{3}$. By the same argument as above we have
\begin{align*}
F_{p}(\zeta_3) &=\left(\frac{6}{p} \right) \left[ (1+ \zeta_3)S_{16}+(\zeta_3^2+\zeta_3^5) S_{26}+(\zeta_3^3+\zeta_2^4) S_{36} \right]\\
                       &=-\zeta_{3}^2 \left(\frac{6}{p} \right) \left[ S_{16}-2S_{26}+S_{36} \right]\\
                       &=3 \zeta_{3}^2 \left(\frac{6}{p} \right) S_{26}.
\end{align*} 
For the last equality, we use the fact that when $p \equiv 1 \pmod{4}$ 
\[ S_{16}+S_{26}+S_{36}=\sum_{a=1}^{\frac{p-1}{2}} \left(\frac{a}{p} \right)=0.\] 
We conclude that 
\[ g_p(-1)=\left(\frac{6}{p} \right) S_{26} .\]
By \cite[Theorem 6.1]{[Berndt]}, we have 
\[ S_{26}=-\frac{1}{2} \left(\frac{2}{p} \right) h(-3p) .\] 
Therefore 
\[ g_p(-1)=-\frac{1}{2} \left(\frac{3}{p} \right) h(-3p) .\]
Let us consider the case $p \equiv 2 \pmod{3}$. Then we have 

\begin{align*} 
F_{p}(\zeta_3) &=\left(\frac{6}{p} \right) \left[ (1+\zeta_3^5)S_{16}+(\zeta_3^4+\zeta_3) S_{26}+(\zeta_3^3+\zeta_3^2) S_{36} \right] \\
                        &=-\zeta_{3} \left(\frac{6}{p} \right) \left[ S_{16}-S_{36} \right]\\
                        &=3 \zeta_{3} \left(\frac{6}{p} \right) S_{26}.
\end{align*} 
By the same argument as above we see that 
\[ g_p(-1)=-\frac{1}{2} \left(\frac{3}{p} \right) h(-3p) .\]
In summary, we have the following proposition. 
\begin{prop}
Let $p \equiv 1 \pmod{4}$, then 
\[ g_p(-1)=-\frac{1}{2} \left(\frac{3}{p} \right) h(-3p).\]
\end{prop}

In order to  summarise all the special values $g_p(u)$ with $u=\pm2,\pm1,0$ achieved so far, we collect the previous propositions into a single theorem as follows.
\begin{thm} Let $p\geq 7$ be a prime number.
\begin{enumerate}
    \item If $p\equiv 3\pmod 4$ then
    \[
    \begin{aligned}
        g_p(2)&=ph(-p),\\
        g_p(-2)=g_p(0)&=-\left(2 \left(\frac{2}{p} \right)-1 \right)h(-p),\\
        g_p(1)&= -\frac{1}{2} \left(\frac{6}{p} \right)  \left[6-3 \left(\frac{2}{p} \right) -2 \left(\frac{3}{p} \right)+\left(\frac{6}{p} \right)\right]h(-p),\\
         g_p(-1)&=-\frac{1}{2} \left(\frac{3}{p} \right) h(-3p).
    \end{aligned}
    \]
    \item If $p\equiv 1\pmod 4$ then
    \[
    \begin{aligned}
        g_p(2)&=\frac{p B_{2,\chi_p}}{4},\\
        g_p(-2)&=-\frac{1}{4} \left( 4 \left(\frac{2}{p} \right)-1 \right) B_{2, \chi_p},\\
        g_p(0)&=-\frac{1}{2} \left(\frac{2}{p} \right) h(-4p)\\
        g_p(1)&=-\frac{1}{2}  \left(2 \left(\frac{2}{p} \right) +1  \right) h(-3p),\\
         g_p(-1)&=-\frac{1}{2} \left(\left(\frac{p}{3} \right)+3 \right) h(-p).
    \end{aligned}
    \]
\end{enumerate}
\end{thm}

Here is a table for the special values of these $g_p(u)$, for $p\leq 23$, at $u=-2, -1, 0, 1, 2$.

\begin{center}
\begin{tabular}{ |c||c|c|c|c|c|c|c|c|} 
 \hline
 $p$ & $g_p(-2)$ & $g_p(-1)$ & $g_p(0)$ & $g_p(1)$ &$g_p(2)$   \\
 \hline
 $7$ &-1 & -2 & -1 & 2 & 7 \\ 
 \hline
 $11$ &3 & -1 & 3 & 3 & 11\\
 \hline 
 $13$ &5 & -2 & 1 & 2& 13\\
 \hline 
 $17$ &-6 & 1 & -2 & -3 & 34 \\
 \hline
 $19$ &3 & -2 & 3 & -6 & 19\\ 
 \hline
 $23$ &-3 & -3 & -3 & -3 & 69\\
 \hline
\end{tabular}
 \end{center}

\section{Galois theory for $f_p$ and $g_p(x)$}  

In this section, we study Galois theory for $f_p$ and $g_p$ for prime $p\geq 7$.  The following lemma is the direct consequence of our previous computations for $g_p(2)$ and $g_p(-2)$.
\begin{lem}
Let $s_p=f_p(1)f_p(-1)=g_p(2)g_p(-2)$. Then 
\[ s_p = \begin{cases}  (1-4 \left(\frac{2}{p} \right))p \left(\frac{B_{2, \chi_p}}{4} \right)^2 &\text { if } p\equiv 1\pmod 4\\
  -(2 \left(\frac{2}{p} \right)-1)p h(-p)^2 &\text { if } p \equiv 3 \pmod 4. \end{cases} \] 
\end{lem} 

\begin{cor} \label{cor:equivalance}
The polynomial $f_p$ is irreducible over $\Q$ if and only if $g_p$ is irreducible over $\Q$.
\end{cor}
\begin{proof}
It is clear that if $g_p$ is reducible over $\Q$ then $f_p$ is reducible over $\Q$.
Now we suppose that $g_p$ is irreducible over $\Q$. 
By the above lemma, we see that $|s_p| = |f_p(1)|\cdot|f_p(-1)|$ is never a square in $\Q$. Hence $|f_p(1)|$ or $|f_p(-1)|$ are not perfect squares. By  \cite[Theorem 11]{[CC]} we conclude that $f_p$ is irreducible.
\end{proof}

We have the following proposition. 
\begin{prop} \label{prop:s_p}
$\sqrt{s_p}$ belongs the splitting field of $f_p$.
\end{prop} 
We provide two proofs for this proposition. The first proof uses the following observation which is interesting on its own. We are grateful to Professor Arturas Dubickas for alerting us that the result was already known. See the following interesting references \cite[page 127]{[AV]}, \cite[page 85]{[Dubickas]} and \cite[page 51]{[CM]}, where the statement was observed and proved. 
\begin{prop} \label{prop: abstract}
Let $f$ be a reciprocal polynomial of even degree $2n$ over a field of characteristics different from $2$. Let $g$ be the polynomial of degree $n$ such that 
\[ f(x)=x^n g(u) ,\]
where $u=x+\dfrac{1}{x}$.  Let $s=(-1)^n f(1) f(-1)$. Then 
\[ \Delta(f)= s \times \Delta(g)^2 ,\]
where $\Delta(f)$ is the discriminant of a monic polynomial $f$. Recall that $\Delta(f)$ is defined to be 
\[ \Delta(f)=\prod_{i < j} (z_i-z_j)^2 ,\]
with $z_i$ are all the roots of $f$. In particular, $\sqrt{s}$ belongs to the splitting field of $f$.

\end{prop} 
\begin{proof}
As above, let $\{u_1, \ldots, u_n \}$ are the roots of $g$. For each $u_i$, the is a corresponding quadratic equation 
\[ u_i=x+\frac{1}{x}. \]
The above equation can be rewritten as 
\[ x^2-u_ix+1 =0 .\] 
Let $x_{i1}, x_{i2}$ be the two roots of this equation. Then, the set $\{x_{i1}, x_{i2} \}_{i=1}^{n}$ is the set of all roots of $f(x)$. We will order this set using the lexicographical order on the product $\{1, 2, \ldots, n \} \times \{1, 2 \}$.  We have the following identities 
\[ x_{i1}+x_{i2}=u_i, x_{i1} x_{i2}=1 , \forall 1 \leq i \leq n .\] 
Let $1 \leq i<j \leq n$, then there are four roots associated with these two indices namely $\{x_{i1}, x_{i2}, x_{j1}, x_{j2} \}$. The term appeared in the discriminant of $f$ associated with these four roots is 
\begin{align*}
(x_{j1}-x_{i1})^2(x_{j1}-x_{i2})^2 (x_{j2}-x_{i1})^2(x_{j2}-x_{i2})^2  &=[(x_{j1}-x_{i1})(x_{j1}-x_{i2})]^2  [(x_{j2}-x_{i1})^2(x_{j2}-x_{i2})]^2 \\
                           &=\left[(x_{j1}^2-u_ix_{j1}+1)(x_{j2}^2-u_i x_{j2}+1) \right]^2.
\end{align*} 
Using the property that $x_{j1} x_{j2}=1$ and $x_{j1}+x_{j2}=u_j$, we can see that 
\begin{align*}
(x_{j1}^2-u_ix_{j1}+1)(x_{j2}^2-u_i x_{j2}+1) &=2-2 u_i (x_{j1}+x_{j2})+x_{j1}^2+x_{j2}^2+u_{j}^2\\
        &= 2-2u_iuj+(u_i^2-2)+u_j^2\\
        &=(u_i-u_j)^2.
\end{align*}

Therefore, we have 
\[ (x_{j1}-x_{i1})^2(x_{j1}-x_{i2})^2 (x_{j2}-x_{i1})^2(x_{j2}-x_{i2})^2 =(u_i-u_j)^4. \] 

Note that when $i=j$, we also have the term 
\[ (x_{i2}-x_{i1})^2=u_i^2-4 .\]

In particular, we have 
\begin{align*}
\prod_{i=1}^{n} (x_{i2}-x_{i1})^2 &=\prod_{i=1}^n(u_i^2-4)=\prod_{i-1}^n (2-u_i)(-2-u_i)\\
                         &= \prod_{i=1}^n(2-u_i) \times \prod_{i=1}^n(-2-u_i)\\
                         &=g(2)g(-2)=(-1)^n f(1)f(-1)=s.
\end{align*}
From these computations, we conclude that 
\[ \Delta(f)=s \times \Delta(g)^2 .\] 

Finally note that $\sqrt{\Delta(f)}$ belongs to the splitting field of $f$. By the above relation, we can conclude that $\sqrt{s}$ belongs to the splitting field of $f$ as well.  

\end{proof}

The second proof is quite similar to the first proof. We actually found the second proof first through some numerical computations with small prime $p$. For the sake of completeness, we include it here. The proof will be almost identical for the two cases $p \equiv 3 \pmod{4}$ and $p \equiv 1 \pmod{4}$. Therefore, we only provide our proof in the case $p \equiv 3 \pmod{4}$.   
\begin{proof}
Let $\Q(f_p)$ be the splitting field of $f_p$.
Let $\{u_i \}_{i=1}^{\frac{p-5}{2}}$ be the roots of $g_p(x)$. By definition of $u$, we know that for all $1 \leq i \leq \frac{p-5}{2}$, we have $u_i \in \Q(f_p)$. 
 Furthermore, for each $1 \leq i \leq \frac{p-5}{2}$, the roots of the following equation are also in $\Q(f_p)$ 
\[ u_i=x+\frac{1}{x}. \]
Hence $\sqrt{u_i^2-4} \in \Q(f_p)$. In particular, $\sqrt{s_p} \in \Q(f_p)$ where 
\[ s_p= \prod_{i=1}^{\frac{p-5}{2}} (u_i^2-4). \] 
Let us compute $s_p$. We notice that 
\[ s_p=\prod_{i=1}^{\frac{p-5}{2}} (u_i^2-4) =\prod_{i=1}^{\frac{p-5}{2}} (2-u_i) \times \prod_{i=1}^{\frac{p-5}{2}} (-2-u_i)=g_p(2)g_p(-2)=f_p(1)f_p(-1). \] 
This completes the proof. 
\end{proof}

We have the following immediate corollaries. 
\begin{cor}
Let $p \equiv 3 \pmod{4}$. Let $h=-(2 \left(\frac{2}{p} \right)-1)p$. Then $\sqrt{h}$ belongs to the splitting field of $f_p$. 
\end{cor} 

\begin{cor}
Let $p \equiv 1 \pmod{4}$. Let $h=(1-4 \left(\frac{2}{p} \right))p$. Then $\sqrt{h}$ belongs to the splitting field of $f_p$. 
\end{cor}

Let us keep the same notations in Proposition  \ref{prop: abstract} and its proof. Furthermore, let $\Q(f)$, $\Q(g)$ be the splitting fields of $f$ and $g$ respectively. Then we have 
\[ \Q(g)=\Q(u_1, \ldots, u_n),\] 
and 
\[ \Q(f)=\Q(g)[x_{i1}, x_{i2}| 1 \leq i \leq n ]. \] 

Note that $x_{i1}, x_{i2}$ are roots of a quadratic equation with coefficients in $\Q(g)$, namely 
\[ x^2-u_ix+1=0 .\] 
We therefore can see that 
\[ [\Q(g)[x_{i1},x_{i2}]:\Q(g)] \leq 2. \] 
Consequently 
\[ [\Q(f):\Q(g)]=[\Q(g)[x_{i1}, x_{i2}| 1 \leq i \leq n ] :\Q(g)] \leq \prod_{i=1}^n [\Q(g)[x_{i1},x_{i2}]:\Q(g)] \leq 2^n.  \] 
The following is an immediate consequence of the above estimate and the fact that $\deg(g)=n$.
\begin{cor} \label{cor:max_Galois}
Let $f,g$ be as in Proposition \ref{prop: abstract},  then 
\[ n! \geq [\Q(g):\Q] \geq \frac{[\Q(f):\Q]}{2^n}. \] 
In particular, if $n!= \dfrac{[\Q(f):\Q]}{2^n}$ then $\Q(g)/\Q$ is a Galois extension with Galois group $S_n$. Additionally, $\Q(f)/\Q$ is a Galois extension of degree $2^n n!$.
\end{cor} 

Using the computer program PARI, we found that for $p \leq 43$, it is always the case that 
\[ [\Q(f_p):\Q]=2^{h_p} (h_p)! ,\]
with $h_p=\dfrac{\deg(f_p)}{2}=\deg(g_p)$. By Corollary  \ref{cor:max_Galois}, we conclude that 
\begin{prop}
Let $p$ be a prime number such that $ p \leq 43$. Then $\Q(g_p)/\Q)$ is a Galois extension with Galois group $S_{h_p}$ where $h_p=\deg(g_p)$. Additionally, $\Q(f_p)/\Q$ is a Galois extension of degree $2^{h_p} (h_p)!$
\end{prop} 
By this proposition, it is reasonable to make the following conjecture. 

\begin{conj} \label{conj1}
 $\Q(g_p)/\Q$ is a Galois extension with Galois group $S_{h_p}$ where $h_p=\deg(g_p)$.
\end{conj}

We provide some further evidence for Conjecture \ref{conj1}. Since it is computationally challenging to compute the degree of $\Q(g_p)$ in general, we develop another strategy to show that $\Q(g_p)/\Q \cong S_{h_p}$ where $h_p=\deg(g_p).$ This strategy is based on the following observation. 

\begin{prop} \label{prop:Galois_computation}
Let $f(x)$ be a monic polynomial with integer coefficients of degree $n$. Assume that there exists a triple of prime numbers $(q_1, q_2, q_3)$ such that 
\begin{enumerate}
    \item $f(x)$ is irreducible in $\F_{q_1}[x]$. 
    \item $f(x)$ has the following factorization in $\F_{q_2}[x]$ 
    \[ f(x)=(x+c)h(x), \]
    where $c \in \F_{q_2}$ and $h(x)$ is an irreducible polynomial of degree $n-1$. 
    \item $f(x)$ has the following factorization in $\F_{q_3}[x]$ 
    \[ f(x)=m_1(x)m_2(x), \]
    where $m_1(x)$ is an irreducible polynomial of degree $2$ and $m_2(x)$ is a product of distinct irreducible polynomials of odd degrees. 
    
\end{enumerate}
Then the Galois group of $\Q(f)/\Q$ is $S_n.$
\end{prop}
A proof for this proposition can be read off from \cite[Example 4.33]{[Milne]} where a particular example of $(q_1, q_2, q_3)$ is discussed. For the sake of completeness, we provide a proof of the this proposition as stated above. 

\begin{proof}
Let $G_f=\Gal(\Q(f)/\Q)$ which is naturally a subgroup of $S_n$. The first condition implies that $f(x)$ is irreducible over $\Z$, hence over $\Q$. By \cite[Proposition 4.4]{[Milne]}, $G_f$ is a transitive subgroup of $S_n$. The second condition implies that $G_f$ contains an $(n-1)$ cycle. The third condition implies that $G_f$ contains a transposition. By \cite[Lemma 4.32]{[Milne]}, we must have $G_f=S_n.$

\end{proof}
\begin{ex}
Let us discuss a concrete example with $p=11$. In this case, we have 
\[ g_{11}(x)=x^4-3x^2+2x+3. \] 
Let $(q_1, q_2, q_3)=(5,7,53)$. Then $g_{11}(x)$ is irreducible in $\F_5[x].$ In $\F_{7}[x]$, $g_{11}(x)$ has the following factorization 
\[ g_{11}(x)= (x + 4)(x^3 + 3x^2 + 6x + 6) .\] 
In $\F_{53}[x]$, $g_{11}(x)$ has the following factorization 
\[ g_{11}(x)=(x + 26) (x + 30)  (x^2 + 50x + 21) .\] 
We see that the triple $(q_1, q_2, q_3)$ satisfies the conditions given in Proposition \ref{prop:Galois_computation}. Therefore the Galois group of $g_{11}(x)$ must be $S_4.$
\end{ex} 

\begin{ex}
Let us consider the case $p=13.$ In this case, we have 
\[ g_{13}(x)=x^4 - 2x^2 + 2x + 1 .\]
Let $(q_1, q_2, q_3)=(3, 5, 61)$. Then $g_{13}(x)$ is irreducible in $\F_{3}[x]$. In $\F_{5}[x]$, it has the following factorization 
\[ g_{13}(x)=(x + 2) (x^3 + 3x^2 + 2x + 3).\] 
In $\F_{61}[x]$, it has the following factorization 
\[ g_{13}(x)=(x + 51)(x + 54) (x^2 + 17x + 34). \] 
We see that $(q_1, q_2, q_3)$ satisfies the conditions given in Proposition \ref{prop:Galois_computation}. We conclude that the Galois group of $g_{13}(x)$ is $S_4.$
\end{ex}
We wrote some SageMath codes to test the above strategy (see the github repository \cite{[codes]} for detailed information about the functionality of our codes). We found that for $p \leq 1600$, the triple $(q_1, q_2, q_3)$ always exists. We provide below the smallest triple $(q_1, q_2, q_3)$ for $p<1000$. We then show the running time for each $p$ in the range $[1000, 1100]$ (Table $3$). Finally, we provide the running time when we search for the triple $(q_1, q_2, q_3)$ for several $p$. As indicated in Table $4$, this is a computationally challenging problem. 
\newpage

\begin{center}
\captionof{table}{Smallest triples $(q_1, q_2, q_3)$ for primes $7< p<500$}
\begin{multicols}{3}

\begin{tabular}{ |c||c|} 
 \hline
 $p$ & $(q_1, q_2, q_3)$ \\ 
 \hline 
 $11$ & $(5, 7, 53)$ \\
 \hline 
 $13$ & $(3,5,61)$ \\
 \hline 
 $17$ & $(19, 3, 11)$ \\
 \hline
 $19$ & $(5, 31, 43)$ \\
 \hline 
 $23$ & $(7, 13, 101)$ \\
 \hline 
 $29$ & $(53, 5, 83)$ \\
 \hline
 $31$ & $(61, 13, 17)$ \\
 \hline 
 $37$ & $(7, 13, 31)$ \\
 \hline 
 $41$ & $(11, 103, 43)$ \\ 
 \hline 
 $43$ & $(5, 31, 23)$ \\
 \hline 
 $47$ & $(107, 7, 53)$ \\ 
 \hline 
 $53$ & $(11, 59, 17)$ \\ 
 \hline 
 $59$ &$(211, 257, 41)$ \\
 \hline
 $61$ & $(197, 5, 41)$ \\
 \hline 
 $67$ & $(113, 41, 29)$ \\
 \hline
 $71$ & $(31, 37, 5)$ \\
 \hline 
 $73$ & $(97, 149, 47)$ \\
 \hline 
 $79$ & $(73, 113, 53)$ \\
 \hline 
 $83$ & $(617, 61, 101)$ \\
 \hline 
 $89$ & $(127, 151, 103)$ \\
 \hline 
 $97$ & $(53, 61, 41)$ \\
 \hline
 $101$ & $(547,149, 89)$ \\
 \hline 
 $103$ & $(457,277,127)$ \\
 \hline 
 $107$ & $(17, 193, 53)$ \\
 \hline
 $109$ & $(127, 293, 157)$ \\
 \hline 
 $113$ & $(23, 491, 101)$ \\
 \hline 
 $127$ & $(223, 197, 41)$ \\
 \hline
 $131$ & $(499, 1193, 19)$ \\
 \hline 
 $137$ & $(839, 523, 59)$ \\
 \hline 
 $139$ & $(673, 103, 157)$ \\ 
 \hline 
 $149$ & $(107, 43, 179)$ \\
 \hline 
 \end{tabular}

\begin{tabular}{ |c||c|} 
 \hline
 $p$ & $(q_1, q_2, q_3)$ \\ 
 \hline 
  $151$ & $(1217, 37, 67)$ \\ 
 \hline 
  $157$ & $(229, 67, 191)$ \\ 
 \hline 
 $163$ &$(23, 239, 103)$ \\
 \hline
 $167$ & $(199, 379, 73)$ \\
 \hline 
 $173$ & $(127, 139, 29)$ \\
 \hline
 $179$ & $(131, 211, 101)$ \\
 \hline 
 $181$ & $(569, 347, 613)$ \\
 \hline 
 $191$ & $(509, 281, 101)$ \\
 \hline
 $193$ & $(13, 307, 107)$ \\
 \hline 
 $197$ & $(2141, 257, 17)$ \\
 \hline 
 $199$ & $(547, 787, 17)$ \\
 \hline
 $211$ & $(47,311, 23)$ \\
 \hline 
 $223$ & $(1481, 179, 103)$ \\
 \hline 
 $227$ & $(317, 439, 223)$ \\
 \hline
 $229$ & $(631, 719, 89)$ \\
 \hline 
 $233$ & $(1559, 977, 29)$ \\
 \hline 
 $239$ & $(199, 17, 59)$ \\
 \hline
 $241$ & $(2857, 1231, 83)$ \\
 \hline 
 $251$ & $(41, 73, 277)$ \\
 \hline 
 $257$ & $(1129,919, 227)$ \\ 
 \hline 
 $263$ & $(1571, 239, 17)$ \\
 \hline 
 $269$ & $(929, 97, 43)$ \\ 
 \hline 
 $271$ & $(821, 3343, 239)$ \\ 
 \hline 
 $277$ &$(317, 2693, 59)$ \\
 \hline
 $281$ & $(283, 131, 71)$ \\
 \hline 
 $283$ & $(89, 953, 199)$ \\
 \hline
 $293$ & $(523, 691, 11)$ \\
 \hline 
 $307$ & $(137, 487, 197)$ \\
 \hline 
 $311$ & $(1291, 2029, 83)$ \\
 \hline 
 $313$ & $(197, 661, 31)$ \\
 \hline 
 $317$ & $(1583, 59, 193)$ \\
 \hline 
\end{tabular}

\begin{tabular}{ |c||c|} 
 \hline
 $p$ & $(q_1, q_2, q_3)$ \\ 
 \hline 
  $331$ & $(53, 1733, 337)$ \\
 \hline
  $337$ & $(3257, 599, 79)$ \\
 \hline
 $347$ & $(2113,173,197)$ \\
 \hline 
 $349$ & $(53, 421, 11)$ \\
 \hline
 $353$ & $(1301, 2689, 653)$ \\
 \hline 
 $359$ & $(1069, 443, 463)$ \\
 \hline 
 $367$ & $(1459, 677, 269)$ \\
 \hline
 $373$ & $(647, 151, 347)$ \\
 \hline 
 $379$ & $(2003, 9421, 337)$ \\
 \hline 
 $383$ & $(47, 59, 71)$ \\ 
 \hline 
 $389$ & $(167, 1423, 401)$ \\
 \hline 
 $397$ & $(701, 5741, 23)$ \\ 
 \hline 
 $401$ & $(1117, 823, 83)$ \\ 
 \hline 
 $409$ &$(59, 157, 107)$ \\
 \hline
 $419$ & $(659, 2939, 149)$ \\
 \hline 
 $421$ & $(1093, 31, 11)$ \\
 \hline
 $431$ & $(163, 2447, 251)$ \\
 \hline 
 $433$ & $(811, 809, 149)$ \\
 \hline 
 $439$ & $(3187, 2143, 593)$ \\
 \hline 
 $443$ & $(5879, 4973, 149)$ \\
 \hline 
 $449$ & $(241, 131, 293)$ \\
 \hline 
 $457$ & $(79, 2393, 233)$ \\
 \hline
 $461$ & $(1531, 3691, 173)$ \\
 \hline 
 $463$ & $(2753, 2999, 97)$ \\
 \hline 
 $467$ & $(463, 593, 113)$ \\
 \hline
 $479$ & $(5527, 1187, 509)$ \\
 \hline 
 $487$ & $(991, 3323, 179)$ \\
 \hline 
 $491$ & $(89, 3347, 103)$ \\
 \hline
 $499$ & $(947, 887, 59)$ \\
  \hline 
 \end{tabular}
 \end{multicols}
\end{center}

\newpage
\begin{center}

\captionof{table}{Smallest triples $(q_1, q_2, q_3)$ for primes $p$: $500< p<1000$} 
\begin{multicols}{3}

\begin{tabular}{|c||c|}

 \hline
 $p$ & $(q_1, q_2, q_3)$ \\ 
 \hline 
 $503$ & $(89, 1913, 19)$ \\
 \hline
 $509$ & $(2729, 617, 71)$ \\
 \hline
 $521$ & $(701, 1069, 277)$ \\
 \hline 
 $523$ & $(541, 3557, 151)$\\
 \hline
 $541$ & $(787, 1553, 109)$ \\
 \hline
 $547$ & $(241, 1049, 73)$ \\
 \hline
 $557$ & $(2027, 271, 131)$ \\
 \hline 
 $563$ & $(593, 929, 107)$ \\
 \hline 
 $569$ & $(4153, 197, 487)$ \\
 \hline
 $571$ & $(79, 683, 71)$ \\
 \hline
 $577$ & $(223, 1759, 229)$ \\
 \hline
 $587$ & $(7457, 2099, 13)$\\
 \hline
 $593$ & $(43, 1367, 439)$\\
 \hline
 $599$ &$(13, 3709, 811)$ \\
 \hline
  $601$ & $(1697, 2459, 103)$ \\
 \hline
 $607$ & $(599, 7207, 211)$ \\
 \hline
 $613$ & $(401, 7559, 331)$ \\
 \hline
 $617$ & $(659, 641, 47)$ \\
 \hline
 $619$ & $(31, 1553, 197)$ \\
 \hline
 $631$ & $(457, 463, 61)$ \\
 \hline
 $641$ & $(751, 6577, 53)$ \\
 \hline
 $643$ & $(5623, 1499, 307)$ \\
 \hline
 $647$ & $(1879, 41, 13)$ \\
 \hline
 $653$ & $(9781, 2711, 19)$ \\
 \hline
 $659$ & $(1543, 5743, 677)$\\
 \hline
 \end{tabular}
 
 \begin{tabular}{|c||c|}
 \hline
  $p$ & $(q_1, q_2, q_3)$ \\ 
 \hline 
 $661$ & $(149, 3469, 233)$ \\
 \hline
 $673$ & $(59, 127, 37)$ \\
 \hline
 $677$ & $(3187, 1451, 97)$ \\
 \hline
 $683$ & $(4603, 3307, 83)$\\
 \hline
 $691$ & $(239, 947, 83)$ \\
 \hline
  $701$ & $(3023, 1231, 29)$\\
 \hline
 $709$ & $(1217, 997, 263)$ \\
 \hline
 $719$ & $(73, 7213, 53)$ \\
 \hline
 $727$ & $(5443, 4111, 43)$ \\
 \hline
 $733$ & $(1367, 3581, 97)$ \\
 \hline
 $739$ & $(4451, 97, 349)$ \\
 \hline
 $743$ & $(359, 13, 37)$ \\
 \hline
 $751$ & $(19, 13267, 601)$ \\
 \hline
 $757$ & $(6421, 491, 97)$ \\
 \hline
 $761$ & $(523, 5281, 5)$ \\
 \hline
 $769$ & $(2099, 2671, 109)$ \\
 \hline
 $773$ & $(10369, 3511, 1061)$ \\
 \hline
 $787$ & $(2861, 251, 443)$ \\
 \hline
 $797$ & $(283, 6091, 7)$ \\
 \hline
 $809$ & $(1009, 5417, 1693)$ \\
 \hline
 $811$ & $(709, 103, 109)$ \\
 \hline
 $821$ & $(2677, 4957, 67)$\\
 \hline
 $823$ & $(443, 8167, 13)$ \\
 \hline
 $827$ & $(9769, 199, 13)$ \\
 \hline
 \end{tabular}
 
 \begin{tabular}{|c||c|}
\hline
 $p$ & $(q_1, q_2, q_3)$ \\ 
 \hline 
 $829$ & $(2017, 2129, 457)$ \\
 \hline
 $839$ & $(41, 1867, 5)$\\
 \hline
 $853$ & $(8017, 4691, 13)$ \\
 \hline
 $857$ & $(919, 3461, 199)$ \\
 \hline
 $859$ & $(1129, 3359, 251)$ \\
 \hline
 $863$ & $(4493, 331, 1151)$ \\
 \hline
 $877$ & $(4999, 1297, 31)$ \\
 \hline
 $881$ & $(1213, 2693, 331)$ \\
 \hline
 $883$ & $(1621, 1889, 97)$ \\
 \hline
 $887$ & $(743, 6547, 29)$ \\
 \hline
 $907$ & $(997, 14767, 277)$ \\
 \hline
 $911$ & $(3931, 4027, 59)$ \\
 \hline
 $919$ & $(839, 9547, 733)$ \\
 \hline
 $929$ & $(4583, 9103, 29)$ \\
 \hline
 $937$ & $(4871, 15467, 3851)$ \\
 \hline
 $941$ & $(3313, 359, 1093)$ \\
 \hline
 $947$ & $(17669, 5641, 223)$ \\
 \hline
 $953$ & $(1973, 4013, 79)$ \\
 \hline
 $967$ & $(859, 4759, 821)$\\
 \hline
 $971$ & $(1973, 1291, 557)$ \\
 \hline
 $977$ & $(2617, 2153, 17)$ \\
 \hline
 $983$ & $(3637, 947, 89)$\\
 \hline
 $991$ & $(239, 3037, 173)$ \\
 \hline
 $997$ & $(4583, 1907, 191)$\\
 \hline
\end{tabular} 
 \end{multicols}

\end{center}

\newpage
We provide some further examples and the running times of our codes. 
\begin{center}

\begin{tabular}{ |p{2cm}||p{4cm}|p{5cm}|} 
 \hline
 $p$ & $(q_1, q_2, q_3)$ & \text{Wall time}\\ 
 \hline
 $1009$ & $(5393, 4211, 593)$ & \text{5min 35s} \\
 \hline
 $1013$ & $(499, 3049, 43)$ & \text{2 min 2s} \\
 \hline
 $1019$ & $(2687, 1373, 193)$ & \text{2min 28s}\\
 \hline
 $1021$ & $(11171, 48187, 79)$ & \text{26min 18s}\\
 \hline
 $1031$ & $(983, 547, 1747)$ & \text{2min}  \\
 \hline
 $1033$  &$(6131, 1789, 79)$ &\text{4min 18s} \\
 \hline
 $1039$ & $(4231, 1367, 383)$ & \text{3min 43s} \\
 \hline
 $1049$ & $(683, 3407, 17)$ &\text{2min 18s} \\
 \hline
 $1051$ & $(859, 1093, 1087)$ & \text{2min 6s} \\
 \hline
 $1061$ &$(2027, 3727, 313)$ & \text{3min 40s} \\
 \hline
 $1063$ & $(2179, 3259, 179)$ & \text{3min 18s} \\
 \hline
 $1069$ & $(1973, 211, 433)$ & \text{1min 43s} \\
 \hline
 $1087$ &$(3863, 1289, 313)$ &\text{3min 25s} \\
 \hline
 $1091$ & $(211, 6301, 311)$ & \text{3min 58s} \\
 \hline
 $1093$ & $(41, 10103, 283)$ & \text{5min 56s} \\
 \hline
 $1097$ & $(1103, 1607, 173)$ &\text{2min 1s} \\
 \hline
\end{tabular}
\captionof{table}{Smallest triples $(q_1, q_2, q_3)$ for primes $1000<p<1100$ and the running times}
\end{center}

Finally, we show the running time when we test several $p$ simultaneously. Here we look for triples $(q_1, q_2, q_3)$ such that $\max \{q_1, q_2, q_3 \} <10^6.$

\begin{center}
\begin{tabular}{ |p{2.5cm}||p{2cm}| p{4cm}|p{3cm}|} 
\hline
 $\text{Interval}$ & \text{\# primes} & \text{Existence of} $(q_1, q_2, q_3)$ & \text{Wall time}\\ 
 \hline
 $(1000, 1100)$ & $16$ & \text{YES} & \text{1h 20min 41s} \\
 \hline
 $(1100, 1200)$ & $11$ &\text{YES} & \text{1h 6min 53s} \\
 \hline
 $(1200, 1300)$ & $15$ &\text{YES} & \text{2h 51min 41s} \\
 \hline
 $(1300, 1400)$ & $11$ &\text{YES} & \text{1h 16min 13s}\\
 \hline
 $(1400, 1500)$ &$17$ &\text{YES} & \text{3h 29min 46s} \\
 \hline
 $(1500, 1600)$ & $12$ &\text{YES} & \text{2h 45min 41s} \\
 \hline

\end{tabular}    
\captionof{table}{Existence of $(q_1, q_2, q_3)$ for primes $1000<p<1600$ and the running times.}

\end{center}

\begin{conj} \label{conj2}
 $\Q(f_p)/\Q$ is a Galois extension with Galois group $(\Z/2\Z)^{h_p}\rtimes S_{h_p}$ where $h_p=\deg(g_p)$ and the symmetric group $S_{h_p}$ acts naturally as the group of permutations on $(\Z/2\Z)^{h_p}$.
\end{conj}
Note that Conjecture \ref{conj2} implies Conjecture \ref{conj1}. We will provide some numerical evidence for Conjecture \ref{conj2}. Since it is computationally difficult to compute the degree of $\Q(f_p)/\Q$ explicitly, we adapt a similar approach as before to show that the Galois group of $\Q(f_p)/\Q$ is $(\Z/2\Z)^{h_p}\rtimes S_{h_p}$ where $2h_p=\deg(f_p).$ This approach is based on the following proposition. 
\begin{prop}
\label{prop:Galois_computation_1}
Let $f(x)$ be a monic reciprocal polynomial with integer coefficients of even degree $2n$. Assume that there exists a quadruple of prime numbers $(q_1, q_2, q_3,q_4)$ such that 
\begin{enumerate}
    \item $f(x)$ is irreducible in $\F_{q_1}[x]$. 
    \item $f(x)$ has the following factorization in $\F_{q_2}[x]$ 
    \[ f(x)=(x+c_1)(x+c_2)h(x), \]
    where $c_1, c_2$ are distinct elements in $\F_{q_2}$ and $h(x)$ is an irreducible polynomial of degree $2n-2$. 
    \item $f(x)$ has the following factorization in $\F_{q_3}[x]$ 
    \[ f(x)=m_1(x)m_2(x), \]
    where $m_1(x)$ is a polynomial of degree $2$ and $m_2(x)$ is a product of distinct irreducible polynomials of odd degrees. 
    \item $f(x)$ has the following factorization in $\F_{q_4}[x]$ 
    \[ f(x)=p_1(x)p_2(x), \]
    where $p_1(x)$ is irreducible polynomial of degree $4$ and $p_2(x)$ is a product of distinct irreducible polynomials of odd degrees. 
\end{enumerate}
Then the Galois group of $\Q(f)/\Q$ is $(\Z/2\Z)^{n}\rtimes S_{n}$
\end{prop}
\begin{proof}
The above conditions shows that the Galois group $\Q(f)/\Q$ contains an $2n$-cycle, an $(2n-2)$-cycle, a $4$-cycle, and a $2$-cycle. By \cite[Lemma 2]{[DDS]}, the Galois group of $\Q(f)/\Q$ is $(\Z/2\Z)^{n}\rtimes S_{n}$.
\end{proof}
Below we provide a table for the existence for $(q_1, q_2, q_3, q_4)$ for $p<600.$

\newpage

{\centering

\captionof{table}{Quadruple $(q_1, q_2, q_3, q_4)$}
\begin{multicols}{3}
\scalebox{0.9}{
\begin{tabular}{|c||c|} 
 \hline
 $p$ & $(q_1, q_2, q_3,q_4)$ \\ 
 \hline 
 $11$ & $(5, 23, 7, 73)$ \\
 \hline
 $13$ & $(3, 19, 31, 103)$ \\
 \hline
 $17$ & $(37, 541, 31, 367)$ \\
 \hline
 $19$ & $(5, 307, 31, 503)$ \\
 \hline
 $23$ & $(7, 97, 181, 241)$ \\
 \hline
 $29$ & $(53, 541, 19, 787)$ \\
 \hline
 $31$ & $(61, 263, 13, 821)$ \\
 \hline
 $37$ & $(7, 19, 109, 53) $\\
 \hline
 $41$ & $(107, 743, 173, 467)$ \\
 \hline
 $43$ & $(7, 751, 1237, 23)$ \\
 \hline
 $47$ & $(107, 419, 421, 409)$ \\
 \hline
 $53$ & $(293, 631, 191, 41)$\\
 \hline
 $59$ & $(211, 1907, 53, 41)$ \\
 \hline
 $61$ & $(197, 89, 487, 2161)$ \\
 \hline
 $67$ & $(257, 227, 167, 337)$ \\
 \hline
 $71$ & $(31, 97, 461, 601)$ \\
 \hline
 $73$ & $(827, 149, 229, 919)$ \\
 \hline
 $79$ & $(691, 173, 113, 71)$ \\
 \hline
 $83$ & $(617, 367, 541, 331)$ \\
 \hline
 $89$ & $(127, 449, 151, 1129)$ \\
 \hline
 $97$ & $(53, 757, 157, 773)$ \\
 \hline 
 $101$ & $(1061, 1213, 149, 89)$ \\
 \hline
 $103$ & $(457, 1013, 211, 4937)$ \\
 \hline
 $107$ & $(797, 211, 139, 3307)$ \\
 \hline
 $109$ & $(373, 467, 293, 797)$ \\
 \hline
 $113$ &$(397, 631, 1217, 1549)$ \\ 
 \hline
 $127$ & $(223, 1811, 97, 53)$ \\
 \hline
 $131$ & $(499, 7549, 1319, 223)$ \\
 \hline
 $137$ & $(839, 9619, 617, 2633)$ \\
 \hline
 $139$ & $(839, 3607, 103, 1801)$ \\
 \hline
 $149$ & $(107, 827, 1823, 5827)$ \\
 \hline
 $151$ & $(1249, 359, 283, 1879)$ \\
 \hline
 $157$ & $(229, 67, 2251, 2609)$ \\
 \hline
 $163$ & $(1879, 13337, 991, 3163)$ \\
 \hline
 $167$ & $(347, 379, 109, 79)$\\
 \hline
 \end{tabular}

 \begin{tabular}{ |c||c|} 
 \hline
 $p$ & $(q_1, q_2, q_3,q_4)$ \\ 
 \hline 
 $173$ & $(503, 409, 191, 3121)$ \\
 \hline
 $179$ & $(421, 2069, 211, 1103)$ \\
 \hline
 $181$ & $(569, 347, 727, 773)$ \\
 \hline
 $191$ & $(509, 2909, 127, 101)$ \\
 \hline
 $193$ & $(13, 307, 673, 4027)$ \\
 \hline
 $197$ & $(5113, 617, 149, 31)$ \\
 \hline
 $199$ & $(547, 787, 8581, 499)$\\
 \hline
 $211$ & $(47, 947, 311, 1439)$ \\
 \hline
 $223$ & $(2333, 3449, 449, 541)$ \\
 \hline
 $227$ & $(317, 3271, 4157, 9001)$ \\
 \hline
 $229$ & $(5519, 719, 1801, 2767)$ \\
 \hline
 $233$ & $(1559, 9601, 2069, 29)$ \\
 \hline
 $239$ & $(199, 809, 233, 179)$ \\
 \hline
 $241$ & $(2857, 1231, 773, 617)$ \\
 \hline
 $251$ & $(41, 433, 443, 277)$ \\
 \hline
 $257$ & $(1129, 5779, 919, 15233)$ \\
 \hline
 $263$ & $(4463, 239, 3769, 11171)$ \\
 \hline
 $269$ & $(929, 6067, 97, 4129)$ \\
 \hline
 $271$ & $(3067, 3343, 4363, 3931)$ \\ 
 \hline
 $281$ & $(3919, 23623, 2089, 1741)$ \\
 \hline
 $283$ & $(89, 8629, 3251, 6691)$ \\
 \hline
 $293$ & $(3373, 1823, 677, 883)$ \\
 \hline
 $307$ & $(353, 487, 661, 557)$ \\
 \hline
 $311$ & $(1523, 8317, 1531, 2347)$ \\
 \hline
 $313$ & $(197, 5849, 263, 947)$ \\
 \hline
 $317$ & $(3769, 1499, 383, 673)$ \\
 \hline
 $331$ & $(53, 1861, 2833, 2081)$ \\
 \hline
 $337$ & $(3257, 599, 4793, 3833)$ \\
 \hline
 $347$ & $(8081, 173, 503, 197)$ \\
 \hline
 $349$ & $(6823, 421, 3329, 2377)$\\
 \hline
 $353$ & $(1301, 4271, 3121, 1831)$ \\
 \hline
 $359$ & $(1069, 9973, 443, 3881)$ \\
 \hline
 $367$ & $(1459, 677, 2113, 2399)$ \\
 \hline
 $373$ & $(5147, 3229, 151, 39113)$ \\
 \hline
 $379$ & $(3061, 9421, 9719, 27043)$ \\
 \hline
\end{tabular}

 \begin{tabular}{ |c||c|} 
 \hline
 $p$ & $(q_1, q_2, q_3,q_4)$ \\ 
 \hline 
  $383$ & $(47, 6217, 59, 2551)$ \\
  \hline
  $389$ & $(857, 16249, 1423, 3221)$ \\
 \hline
 $397$ & $(3323, 5741, 7459, 3061)$ \\
 \hline
 $401$ & $(2729, 8269, 823, 11287)$\\
 \hline
 $409$& $(2969, 4229, 157, 3559)$\\
 \hline
 $419$& $(659, 8537, 3307, 7369)$\\
 \hline
 $421$& $(3637, 431, 6983, 59)$\\
 \hline
 $431$ & $(1579, 2447, 2621, 601)$\\
 \hline
 $433$ & $(811, 2347, 5087, 2311)$\\
 \hline
 $439$ & $(3187, 2143, 1997, 4129)$\\
 \hline
 $443$ &  $(5879, 4973, 10597, 7487)$\\
 \hline
 $449$ & $(241, 5419, 43, 3217)$\\
 \hline
 $457$ & $(10859, 13009, 47, 3229)$\\
 \hline
 $461$ & $(1531, 3691, 269, 211)$\\
 \hline
 $463$ & $(2753, 5119, 2087, 3347)$\\
 \hline
 $467$ & $(463, 12853, 1493, 9661)$\\
 \hline
 $479$ & $(5527, 5471, 1187, 3307)$\\
 \hline
 $487$ & $(991, 14051, 7477, 2837)$\\
 \hline
 $491$ & $(461, 12721, 3347, 1867)$\\
 \hline
 $499$ & $(4397, 14653, 4937, 6197)$\\
 \hline
  $499$ & $(4397, 14653, 4937, 6197)$\\
  \hline
 $503$ & $(89, 1913, 307, 19)$\\
 \hline
 $509$ &  $(2729, 5483, 4201, 337)$\\
 \hline
 $521$ & $(701, 1069, 1747, 19379)$\\
 \hline
 $523$ & $(541, 5113, 2657, 12893)$\\
 \hline
 $541$ & $(4463, 1871, 367, 6761)$\\
 \hline
 $547$ & $(241, 1861, 1049, 3967)$\\
 \hline
 $557$ & $(4409, 271, 977, 5519)$\\
 \hline
 $563$ & $(593, 10181, 953, 15053)$\\
 \hline
 $569$ & $(7673, 5839, 197, 6803)$\\
 \hline
 $571$ & $(79, 1567, 1873, 2333)$\\
 \hline
 $577$ & $(421, 1759, 11177, 947)$\\
 \hline
 $587$ & $(7457, 17921, 17029, 13)$\\
 \hline
 $593$ & $(43, 2767, 16193, 12689)$\\
 \hline
 $599$& $(1597, 3709, 7829, 23743)$\\
 \hline
\end{tabular}  
}
\end{multicols}
}

Finally, we provide the running times for some larger primes $p$. 
\vspace{2em}
\begin{center}
\begin{tabular}{ |p{2cm}||p{5cm}|p{4cm}|} 
 \hline
 $p$ & $(q_1, q_2, q_3, q_4)$ & \text{Wall time}\\ 
 \hline
 $601$  & $(9181, 4691, 499, 12409)$ & \text{15min 8s} \\
 \hline
 $607$  & $(599, 7207, 7541, 9463)$ & \text{15min 8s} \\
 \hline
 $613$  & $(401, 27901, 1109, 7853)$ & \text{24min 4s} \\
 \hline
 $617$  & $(7307, 53731, 10597, 11171)$ & \text{54min 37s} \\
 \hline
 $619$ & $(2039, 1553, 1051, 6221)$ & \text{6min 9s} \\  
 \hline
 $631$  & $(57329, 463, 359, 10847)$ & \text{47min 3s} \\
 \hline
\end{tabular}  
\captionof{table}{Smallest triples $(q_1, q_2, q_3, q_4)$ for primes $600<p<632$ and the running times}

\end{center}

\begin{rmk}
A consequence of Conjecture \ref{conj2} is that the Fekete polynomial $f_p(x)$ is irreducible over $\Z[x].$ While it is computationally expensive to verify the full strength of Conjecture $\ref{conj2}$, it is easier to test the irreducibility of $f_p(x)$. By Corollary \ref{cor:equivalance}, over $\Z[x]$, the irreducibility of $f_p(x)$ is equivalent to the irreducibility of $g_p(x).$ We have verified that $g_p(x)$ and $f_p(x)$ are irreducible for $p<10^4$. We remark that to test for irreducibility, we use the built-in function \textit{is\_irreducible()} in Sagemath instead of using the factorization of $g_p(x)$ over finite fields. This method returns the results faster. For example, it took less than two minutes to verify the irreducibility of $g_p(x)$ for $1000<p<2000$. For further details, we refer to \cite{[codes]}.

We do not have a precise explanation for the irreducibility of $f_p(x)$ and $g_p(x)$. However, the data seems to suggest that $f_p(x)$ behaves like a random reciprocal polynomial (see \cite{[DDS]} for a detailed study of random reciprocal polynomials.) 

\end{rmk}

\section{Modular properties of $f_p(x)$ and $g_p(x)$}
First, we study the reduction of $F_p(x)$ modulo $p$.
\begin{prop} \label{mod_p}
The reduction modulo $p$ of $F_p(x)$ has the following factorization in $\F_p[x]$
\[ F_{p}(x)=(x-1)^{\frac{p-1}{2}} h(x), \]
where $h(x)$ is a polynomial in $\F_p[x]$ and $h(1) \neq 0$.
\end{prop} 
\begin{proof}
Because the degree of $F_p(x)$ is less than $p$, the above statement is equivalent to the following conditions (these equations are taken in $\F_p[x]$). 
\begin{enumerate}
\item For $r < \frac{p-1}{2}$, $F_{p}^{(r)}(1)=0$ where $F_{p}^{(r)}(x)$ is the $r$-th derivative of $F_p(x)$.
\item $F_{p}^{(\frac{p-1}{2})}(1) \neq 0$.

\end{enumerate} 

By definition, the $r$-the derivative of $F_p(x)$ is given by 
\[ F_{p}^{(r)}(x)=\sum_{a=1}^{p-1} \left(\frac{a}{p} \right) a(a-1)\ldots (a-r+1) x^{a-r} .\] 
For example when $r=1$ 
\[ F_{p}^{(r)}(x)=\sum_{a=1}^{p-1} \left(\frac{a}{p} \right) a x^{a-1} .\] 

The leading term of $a(a-1)\cdots (a-r+1)$ is $a^r$. We can see that the above two conditions are equivalent to the following two conditions. 

\begin{enumerate}
\item For $r <\frac{p-1}{2}$ 
\[ \sum_{a=1}^{p-1} \left(\frac{a}{p} \right) a^r \equiv 0 \pmod{p} .\] 
\item For $r=\frac{p-1}{2}$ 
\[ \sum_{a=1}^{p-1} \left(\frac{a}{p} \right) a^r \not \equiv 0 \pmod{p} .\] 

\end{enumerate} 
We can prove (1) and (2) as follows: By Euler's criterion, $\left(\dfrac{a}{p}\right)\equiv a^{\frac{p-1}{2}}\pmod p$, for $1\leq a\leq p-1$. By considering modulo $p$, one has
\[
\sum_{a=1}^{p-1} \left(\frac{a}{p} \right) a^r\equiv \sum_{a=1}^{p-1} a^{\frac{p-1}{2}+r} \equiv \begin{cases} 0 &\text{ if } r<\frac{p-1}{2}\\
-1 &\text{ if  } r=\frac{p-1}{2}.
\end{cases}
\]
\end{proof}

We define  
\[r_p= \begin{cases} 1 &\mbox{if } p \equiv 3 \pmod{4} \\
2 & \mbox{if } p \equiv 1 \pmod{4}. \end{cases}
\] In other words, $r_p$ is  the multiplicity of the root $x=1$ of $F_p(x)$.
Then by Proposition \ref{mod_p}, we have the following corollary. 
\begin{cor}
\label{cor: Fekete mod p}
$f_p(x)$ has the following factorization in $\F_p[x]$ 
\[ f_p(x)=(x-1)^{\frac{p-1}{2}-r_p} g(x) ,\] 
where $g(1) \neq 0$. 
\end{cor} 

If $p \geq 7$ then $\frac{p-1}{2}-r_p \geq 2$. Therefore, if $p \geq 7$ then $x=1$ is a multiple root of $f_p(x)$. We have the following immediate corollary. 
\begin{cor}
If $p \geq 7$ then $p\mid \Delta(f_p)$ where $\Delta(f_p)$ is the discriminant of $f_p$.
\end{cor} 
In fact, the following stronger statement holds. Here, for a prime number $q$ and an integer $n$, $v_q(n)$ is the $q$-adic valuation of $n$.

\begin{cor}
For all primes $p$ we have 
\[ v_p(\Delta(f_p)) \geq \frac{p-3}{2}-r_p .\] 

\end{cor} 
\begin{proof}
By Corollary~\ref{cor: Fekete mod p}, the degree of the greatest common divisor of the reductions of $f_p$ and
$f_p^\prime$ modulo $p$ is at least $\frac{p-3}{2}-r_p$. Now the statement follows from \cite[Theorem]{GGIS}.
\end{proof}

We also have the following estimation at the prime $q=2$. 

\begin{prop} One has $v_2(\Delta(f_p))\geq \deg f_p$.
\end{prop}
\begin{proof}
Let $f(x)= \sum\limits_{k=0}^{p-3}a_kx^k=\dfrac{F_p(x)}{x(1-x)}$. Then by using the relation 
\[
(1-x)\sum\limits_{k=0}^{p-3}a_kx^k=\sum_{a=1}^{p-2}\left(\dfrac{a}{p}\right) x^{a-1},
\]
 one has, for $k=0,\ldots,p-3$,
\[
a_k= \sum_{a=1}^{k+1}\left(\dfrac{a}{p}\right).
\] Thus, if $k$ is odd then $a_k\equiv \sum_{a=1}^{k+1} 1\equiv 0\pmod 2$. This implies that $f'(x)\equiv 0\pmod 2$.
Suppose further that $p\equiv 1\pmod 4$. In this case
$f(x)=(1-x)(x+1)f_p(x)$. Thus 
\[
f'(x) =-(x+1)f_p(x) + (1-x)f_p(x) +(1-x)(1+x)f_p'(x) \equiv (x+1)^2 f_p'(x) \pmod 2
\]
This implies that $f_p'(x)\equiv 0\pmod 2$. 
Therefore for any prime number $p$, one always has $f'_p(x)\equiv 0\pmod 2$ and hence  and $f_p'(x)=2h(x)$ for some $h(x)\in \Z[x]$.
Now, one has
\[
\Delta(f_p) =(\pm )R(f_p,f'_p)=(\pm) R(f_p, 2h)=(\pm)2^{\deg f_p} R(f_p,h).
\]
Here, $R(f_p,f'_p)$ is the resultant of $f_p$ and $f'_p$.
This implies that $v_2(\Delta(f_p))\geq \deg f_p$.
\end{proof}

\begin{rmk} The inequality in the above proposition could be strict. For example, for $p=19$, $v_2(\Delta(f_{19}))=18>16=\deg f_{19}$. 
\end{rmk}

\section*{Acknowledgments}
The first-named author would like to thank Professor Paulo Ribenboim for many discussions concerning class numbers of algebraic number fields and properties of Bernoulli numbers. The second-named author would like to thank Professor Kazuya Kato on some helpful discussions on $p$-adic $L$-functions. He is thankful to Professor David Harvey for his help with the numerical computations of irregular primes.  He is also grateful to Ricardo Buring for his expertise and help with Sagemath. The third-named author gratefully acknowledges the Vietnam Institute for Advanced Study in Mathematics (VIASM) for hospitality and support during a visit in 2021. We would like to thank Professor  Art\={u}ras Dubickas for his interest in our paper and for sending us references on Proposition 4.3, which was proved earlier. We thank Professor Franz Lemmermeyer for his nice comments on our paper after we posted it on arXiv and for sending us interesting references. 
We also thank Professor Danny Neftin for his encouragement. We thank Professor William Duke who kindly sent to us a copy of the paper \cite{[DDS]} which we used in Proposition \ref{prop:Galois_computation_1}.  We would like to thank the Editor of the Journal of Number theory for his kind encouragement and valuable suggestions about adding further numerical evidence towards our conjectures on the Galois groups of Fekete polynomials. 
Last but not least, we are also grateful to the referee for his/her comments and valuable suggestions which we have used to improve our exposition.


\begin{thebibliography}{9999}

\bibitem{[Alex]}
{G. L. Alexanderson,  The random walks of George Pólya.  MAA Spectrum, Mathematical Association of America, Washington, DC, 2000, 213--215, Appendix 3, written by P. H. Lehmer.}

\bibitem{[AV]}
{ O. Ahmadi, G. Vega, On the parity of the number of irreducible factors of self-reciprocal polynomials over finite fields, Finite Fields Appl. 14 (2008), no. 1, 124--131.} 

\bibitem{[Apostol]} 
{M. Apostol, Introduction to analytic number theory, Undergraduate Texts in Mathematics. Springer-Verlag, New York-Heidelberg, 1976.}

\bibitem{[Bateman]}
{P. T. Bateman, G. B. Purdy, S. S. Wagstaff, Some numerical results on Fekete polynomials, Mathematics of Computation 29 (1975), 7--23.}

\bibitem{[Berndt]} 
{B. Berndt, Classical theorems on quadratic residues, Enseign. Math. 22 (1976), 261--304.}

\bibitem{[CC]}
 { A. Cafure,  E. Cesaratto,  Irreducibility criteria for reciprocal polynomials and applications, Amer. Math. Monthly 124 (2017), no. 1, 37--53.}

\bibitem{[Carlitz]}
{L. Carlitz. Some sums connected with quadratic residues, Proc. Amer. Math. Soc. 4 (1953), 12--15.} 

\bibitem{[CM]}
{C. Christopoulos, J. McKee, Galois theory of Salem polynomials, Math. Proc. Cambridge Philos. Soc. 148 (2010), 47--54.}

\bibitem{[Bloch-Kato]} 
{J. Coates, A. Raghuram, A. Saikia, R. Sujatha (eds.), The Bloch-Kato conjecture for the Riemann zeta function, London Mathematical Society Lecture Note Series 418, Cambridge University Press, Cambridge, 2015.}



\bibitem{[Conrey]}
{B. Conrey, A. Granville, B.  Poonen, K. Soundararajan, Zeros of Fekete polynomials,  Annales de l'institut Fourier 50 (2000), no. 3,  865--889.}

\bibitem{[DDS]}
{S. Davis,  W. Duke, X. Sun, Probabilistic Galois theory of reciprocal 
polynomials, Exposition. Math. 16 (1998), 263--270.}

\bibitem{[Davenport]} 
H. Davenport, Multiplicative Number Theory, Second Ed., Graduate Texts in Mathematics 74, Springer-Verlag, New York-Berlin, 1980.
\bibitem{[Dubickas]}
 {A. Dubickas,  Salem numbers as Mahler measures of nonreciprocal units, Acta Arith. 176 (2016), no. 1, 81--88.}

\bibitem{[Gauss]}
{C. F. Gauss, Theorematis fundamentalis in doctrina de residuis quadraticis demonstrationes et amplicationes novae, 1818; Werke II, 47--64.}


\bibitem{[ClassNumber]} 
{K. Girstmair,  A popular class number formula, Amer. Math. Monthly 101 (1994), no. 10, 997--1001.} 

\bibitem{GGIS} D. Gomez, J. Gutierrez, A. Ibeas, D. Sevilla,
Common factors of resultants modulo $p$, 
Bull. Aust. Math. Soc. 79 (2009), no. 2, 299--302. 

\bibitem{[Irregular]} 
{ W. Hart, D. Harvey, W. Ong, Irregular primes to two billion,
Math. Comp. 86 (2017), no. 308, 3031--3049.} 


 \bibitem{[Iwasawa]} 
{K. Iwasawa, Lectures on $p$-adic $L$-functions, 
Annals of Mathematics Studies 74, Princeton University Press, Princeton, N.J.; University of Tokyo Press, Tokyo, 1972. }

\bibitem{[Kurihara]}
{M. Kurihara,  Some remarks on conjectures about cyclotomic fields and K-groups of $\Z$,
Compositio Math. 81 (1992), no. 2, 223--236.}

\bibitem{[Lehmer]}
{E. Lehmer, On congruences involving Bernoulli numbers and the quotients of Fermat and Wilson, Ann. of Math. (2) 39 (1938), no. 2, 350--360.}

\bibitem{[Lemmermeyer]}
{F. Lemmermeyer, Quadratic number fields, Springer Undergraduate Mathematics Series, Springer, 2021.}

\bibitem{[Milne]}
{J. Milne, Fields and Galois theory, (2020), \url{https://www.jmilne.org/math/CourseNotes/ft.html}}

\bibitem{[MTT1]}
{J\'an Min\'a\v{c}, Tung T. Nguyen, Nguy$\tilde{\text{\^{e}}}$n Duy T\^{a}n, Further insight into mysteries of values of zeta functions at integers, preprint, available at \url{https://arxiv.org/abs/2108.08171}} 

\bibitem{[codes]}
{J\'an Min\'a\v{c}, Tung T. Nguyen, Nguy$\tilde{\text{\^{e}}}$n Duy T\^{a}n, Github repository for the codes, \url{https://github.com/tungprime/Fekete-polynomials-Calculations}}

\bibitem{[Po1]}
{G. P\'olya, George Collected papers. Vol. II: Location of zeros, edited by R. P. Boas, Mathematicians of Our Time, vol. 8, the MIT Press, Cambridge, Mass.-London, 1974, 1--26.}

\bibitem{[Po2]}
{G. P\'olya, Verschiedene Bemerkung zur Zahlentheorie, Jber. deutsch Math. Verein 28 (1919), 31--40.}

\bibitem{[Po3]}
{G. P\'olya and G. Szego, Problems and theorems in analysis, v. 2: Theory of functions, zeros, polynomials, determinants, number theory, geometry, revised and enlarged translation of the 4th German edition, Grundlehren Math. Wiss 216.}

\bibitem{[Sh1]} 
{ G. Shimura, Elementary Dirichlet series and modular forms, Springer Monographs in Mathematics. Springer, New York, 2007.}

\bibitem{[Sh2]} 
{G. Shimura, The critical values of generalizations of the Hurwitz zeta function,  Doc. Math. 15 (2010), 489-–506.}

\bibitem{[Washington]} 
L. C. Washington, Introduction to cyclotomic fields, second edition, Graduate Texts in Mathematics 83, Springer-Verlag, New York, 1997.

\bibitem{[Weibel]}
{C. Weibel,  An introduction to algebraic $K$-theory, Graduate Studies in Mathematics 145, American Mathematical Society, Providence, RI, 2013.}





\end{thebibliography}
\end{document}